\documentclass[11pt,a4paper,twoside]{article}
\RequirePackage{filecontents}
\usepackage[utf8]{inputenc}
\usepackage[UKenglish]{babel}
\usepackage[T1]{fontenc}
\usepackage{fancyhdr}
\usepackage[left=2cm,right=2cm,top=2cm,bottom=2cm]{geometry}
\usepackage{etex}
\usepackage[etex=true,export]{adjustbox}
\usepackage{float}
\usepackage{xkeyval}
\usepackage{mathtools}
\usepackage{amsmath}
\usepackage{amsthm}
\usepackage{amsfonts}
\usepackage{amssymb}
\usepackage{amssymb, amsfonts, amsmath, amsthm, times}
\usepackage{epsfig}
\usepackage{lipsum}
\usepackage{graphicx}
\usepackage[sort, numbers]{natbib}
\usepackage{placeins}
\usepackage{hyperref}
\usepackage[dvipsnames]{xcolor}
\usepackage{ulem}
\usepackage{bbm}
\usepackage{authblk}

\newcommand{\Z}{\mathbbm{Z}}

\theoremstyle{definition}
\newtheorem{mydef}{Definition}[section]
\newtheorem{rmq}{Remark}[section]
\newtheorem{empl}{Example}[section]
\newtheorem{conj}{Conjecture}[section]
\theoremstyle{plain}
\newtheorem{theo}{Theorem}[section]
\newtheorem{lem}[theo]{Lemma}
\newtheorem{prop}[theo]{Proposition}
\newtheorem{cor}[theo]{Corollary}
\newtheorem{qst}{Question}

\binoppenalty=\maxdimen
\relpenalty=\maxdimen
\title{Generating links that are both quasi-alternating and almost alternating}
\author[1]{Hamid Abchir}
\author[2]{Mohammed Sabak}
\affil[1]{Hassan II University. EST. Route d'El Jadida Km 7. B.P. 8012. 20100 Casablanca. Morocco.
	{e-mail: hamid.abchir@univh2c.ma}}
\affil[2]{Hassan II University. Ain Chock Faculty of sciences. Route d'El Jadida Km 8. B.P. 5366. Maarif. 20100 Casablanca. Morocco. {e-mail: mohammed.sabak-etu@etu.univh2c.ma}}
\begin{document}
\maketitle
\begin{abstract}
We construct an infinite family of links which are both almost alternating and quasi-alternating from a given either almost alternating diagram representing a quasi-alternating link, or connected and reduced alternating tangle diagram. To do that we use what we call a dealternator extension which consists in replacing the dealternator by a rational tangle extending it. We note that all not alternating and quasi-alternating Montesinos links can be obtained in that way. We check that all the obtained quasi-alternating links satisfy Conjecture 3.1 of Qazaqzeh et al. (JKTR 22 (06), 2013), that is the crossing number of a quasi-alternating link is less than or equal to its determinant. We also  prove that the converse of Theorem 3.3 of Qazaqzeh et al. (JKTR 24 (01), 2015) is false.
\end{abstract}

\section{Introduction}
The set of quasi-alternating links appeared in the context of link homology as a natural generalization of alternating links. They were defined in \cite{ozsvath2005heegaard} where the authors showed that they are homologically thin for both Khovanov homology and knot Floer homology as alternating links with which they share many properties. On the other hand, it was shown in \cite{ozsvath2005heegaard} that every non-split alternating link is quasi-alternating and that the branched double cover of any quasi-alternating link is an $L$-space.\\
If $D$ is a link diagram, we denote by $\mathcal{L}(D)$ the link for which $D$ is a projection. Quasi-alternating links are defined recursively as follows:
\begin{mydef}
 The set $\mathcal{Q}$ of \textbf{quasi-alternating links} is the smallest set of links satisfying the following properties:
\begin{enumerate}
\item The unknot belongs to $\mathcal{Q}$,
\item If $L$ is a link with a diagram $D$ containing a crossing $c$ such that 
\begin{enumerate}
\item for both smoothings of the diagram $D$ at the crossing $c$ denoted by $D^c_0$ and $D^c_\infty$ as in Fig. \ref{fig.1}, the links $\mathcal{L}(D^{c}_{0})$ and $\mathcal{L}(D^{c}_{\infty})$ are in $\mathcal{Q}$ and,
\item $\det(L)=\det(\mathcal{L}(D^{c}_{0}))+\det(\mathcal{L}(D^{c}_{\infty})).$
\end{enumerate}
Then $L$ is in $\mathcal{Q}$. In this case we will say that $c$ is a
\textit{quasi-alternating crossing} of $D$ and that $D$ is quasi-alternating at $c$.
\end{enumerate}

\end{mydef}
\begin{figure}[H]
\centering
\includegraphics[width=0.4\linewidth]{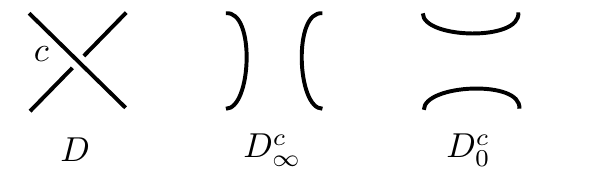}
\caption{The link diagram $D$ and its smoothings $D^{c}_{0}$ and $D^{c}_{\infty}$ at the crossing $c$.}
\label{fig.1}
\end{figure}

Champanerkar and Kofman proved that quasi-alternating property is inheritable via \textit{rational extension} of a quasi-alternating crossing \cite{champanerkar2009twisting}, that is the operation which consists in replacing a quasi-alternating crossing of a diagram by a rational tangle extending it as in Fig. \ref{fig10}.\\
Champanerkar and Kofman proved the following theorem.

\begin{theo}[Theorem 2.1, \cite{champanerkar2009twisting}]
  If $D$ is a quasi-alternating link diagram, let $D^{c \leftarrow T}$ be the link diagram obtained by replacing any quasi-alternating crossing $c$ with an alternating rational tangle $T$ that extends $c$. Then $D^{c \leftarrow T}$ is quasi-alternating.
\label{th1.1}
\end{theo}
Thus the last theorem provides a way to get new quasi-alternating diagrams from former ones.\\
Note that when one extends a quasi-alternating crossing, all crossings of the inserted tangle become quasi-alternating also, including the extended crossing itself.\\
In this paper we will give an other way to build new quasi-alternating diagrams also relying on rational extensions. Nevertheless, unlike what was done in \cite{champanerkar2009twisting}, the crossing which will be extended will not be quasi-alternating. To do that, we will start with an almost alternating link diagram $D$, i.e. a diagram in which one crossing change makes it alternating. We suppose that $\mathcal{L}(D)$ is quasi-alternating. Then we consider a crossing the change of which makes the diagram alternating. Such a crossing is called a \textit{dealternator} of $D$. Note that an almost alternating diagram can have more than one dealternator. But that cannot occur when that diagram is the projection of a quasi-alternating link (Proposition 2.2, \cite{tsukamoto2004criterion}). Our first observation is that a dealternator cannot be a quasi-alternating crossing as mentionned in Corollary \ref{cor2}. However, we show that there is a rational extension of $D$ at the dealternator which generates a quasi-alternating diagram. We call such operation a \textit{dealternator extension} of $D$. We get the following Theorem which is one of our main results.

\begin{theo}
Let $D$ be an almost alternating diagram representing a quasi-alternating link. If $c$ is the dealternator of $D$, then there exists a dealternator extension of $D$ where $c$ becomes a quasi-alternating crossing.
\label{mytheo1}
\end{theo}

The proof will be given in Section \ref{sect5} after a more detailed statement (Theorem \ref{th3.2}).\\

On the other hand, we show that all non-alternating quasi-alternating Montesinos links can be obtained by some dealternator extensions of rational links. In Corollary \ref{mycor4}, we check that any quasi-alternating link arising as a dealternator extension of an almost alternating diagram which represents a quasi-alternating link satisfies the following conjecture formulated by Khaled Qazaqzeh, Balqees Qublan and Abeer Jaradat in \citep{qazaqzeh2013remark} and which compares the crossing number $c(L)$ of a quasi-alternating link $L$ to its determinant $\det(L)$.

\begin{conj}
Every quasi-alternating link $L$ satisfies $c(L) \leq \det(L)$.
\label{conj1}
\end{conj}

The paper is organized as follows. In the second section we recall the main tools needed to prove our results, mainly graphs and tangles. We give a brief overview on the link families which we will deal with. We recall some of their properties. In the third section, we show Theorem \ref{mytheo1} where the technique of dealternator extension is introduced and show some relative results. In the fourth section, we give some applications of Theorem \ref{mytheo1}. We describe a method of generating links which are both almost alternating and quasi-alternating. Then we show how that process allows to construct all non-alternating quasi-alternating Montesinos links. The fifth section is devoted to the consolidation of Conjecture \ref{conj1}. In the last section we answer the question asked by K. Qazaqzeh, N. Chbili and B. Qublan at the end of \cite{qazaqzeh2015characterization}.

\section{Preliminaries}
\subsection{Graphs}
To prove some of our results, we will need some graph-theoretical machinery which will be found in \citep{thistlethwaite1987spanning}.\\
For any connected link diagram $D$, we can associate a connected graph $G(D)$, called the Tait graph of $D$ by checkerboard coloring complementary regions assigning a vertex to every shaded region, an edge to every crossing and a $\pm$ sign to every edge according to the convention in Fig. \ref{fig1}. Note that there are two choices for the checkerboard coloring which give dual graphs.

\begin{figure}[H]
\centering
\includegraphics[width=0.15\linewidth]{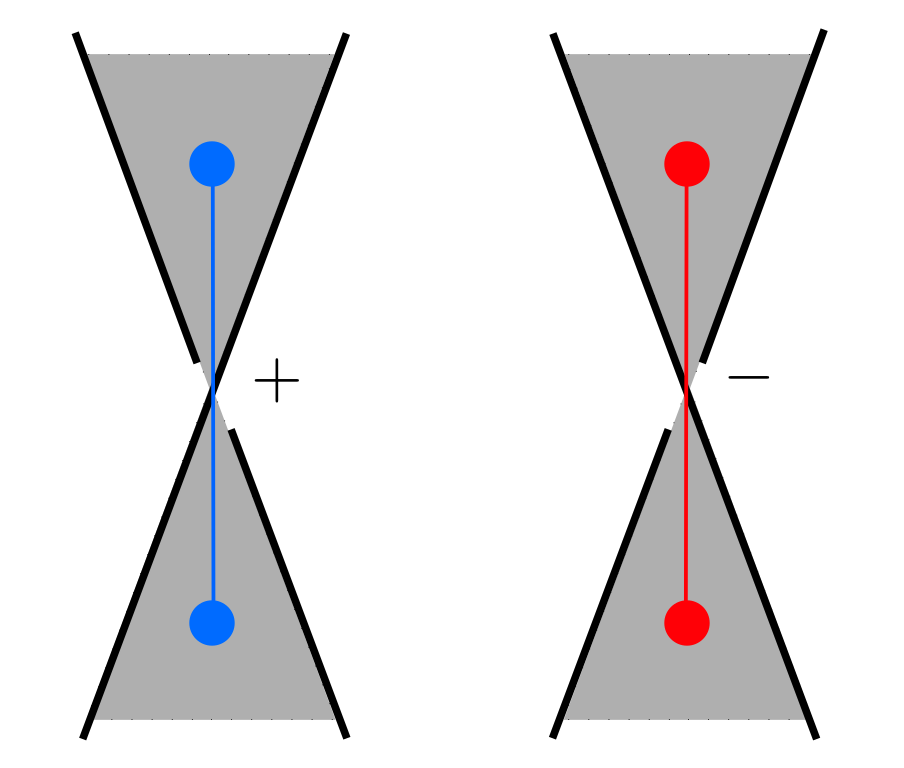}  
\caption{The sign convention for the Tait graph of a link diagram.} 
\label{fig1}
\end{figure}

\begin{figure}[H]
\centering
\includegraphics[width=0.2\linewidth]{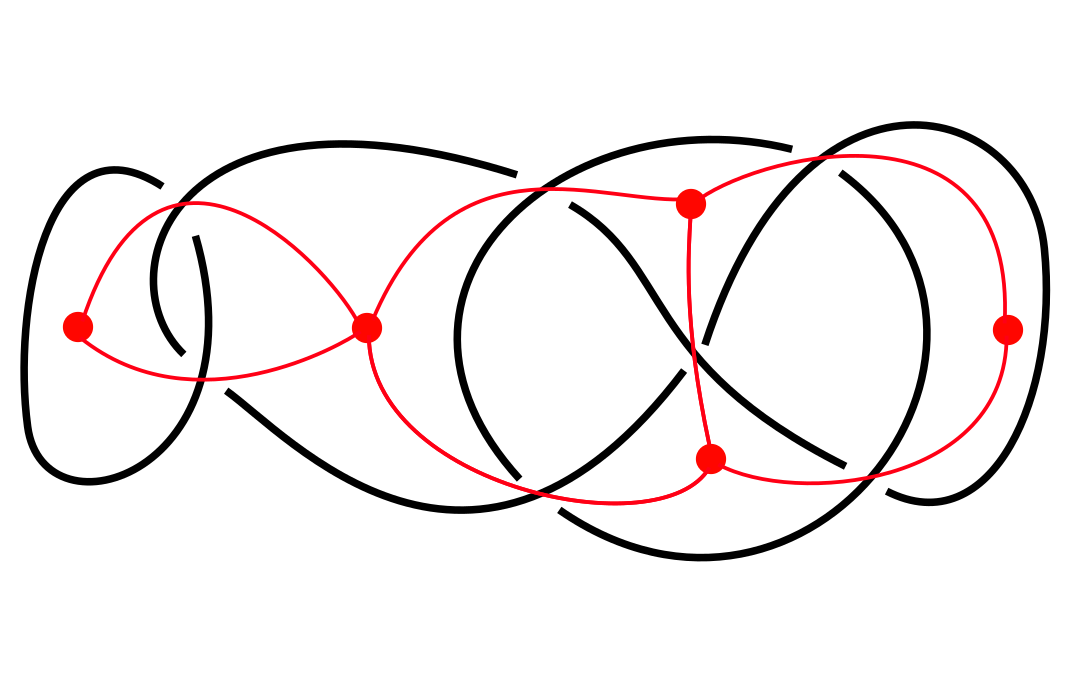}  
\caption{A Tait graph with only negative edges.} 
\label{nfig3}
\end{figure}

\begin{figure}[H]
\centering
\includegraphics[width=0.3\linewidth]{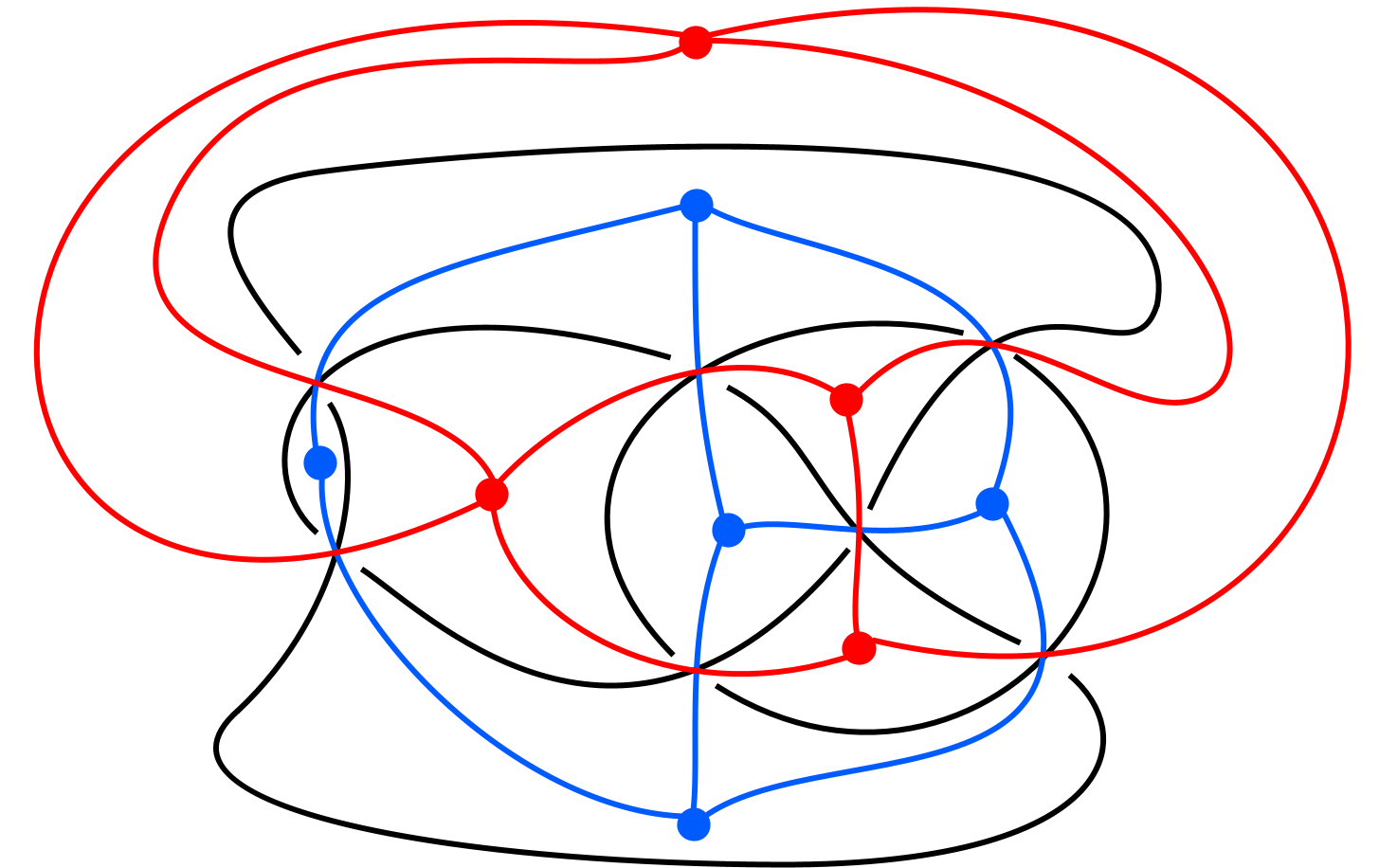}  
\caption{A Tait graph with only positive edges (in blue) and its dual (in red).}
\label{nfig4} 
\end{figure}

Since an edge and its dual have opposite signs, we will always choose the Tait graphs which have more positive edges than negative ones. Note that the edges of a Tait graph of an alternating link diagram are all of the same sign.\\

Graphs allow to get some link invariants like the determinant. A. Champanerkar and I. Kofman showed the following lemma \citep{champanerkar2009twisting}.
\begin{lem}\label{th2.1} For any spanning tree $T$ of $G(D)$, let $v(T)$ be the number of positive edges in $T$.\\
  Let $s_k(D) = \# \lbrace \text{ spanning trees}\ $T$\ of\ G(D)\ |\ v(T)=k\rbrace$. Then
$$\det(D) = \left| \displaystyle \sum_k (-1)^k s_k(D) \right|.$$
\end{lem}

\begin{rmq}
In particular, the determinant of an alternating link is the number of the spanning trees in a Tait graph of any alternating diagram of that link.
\end{rmq}

\subsection{Tangles}
In this paper, we call a \textit{tangle} $T$ any proper embedding of two disjoint arcs and a (possibly empty) set of loops in a $3$-ball $B^3$. Two tangles $T$ and $T^\prime$ are \textit{equivalent} if there is an ambient isotopy of $B^3$ which is the identity on the boundary and which takes $T$ to $T^\prime$.\\
We assume that the four endpoints lie in the great circle of the boundary sphere of $B^3$ which joins the two poles. That great circle bounds a two disk $B^2$ in $B^3$. We consider a regular projection of $B^3$ on $B^2$. The image of a tangle $T$ by that projection in which the height information is added at each of the double points is called a \textit{tangle diagram} of $T$. Two tangle diagrams will be equivalent if they are related by a finite sequence of planar isotopies and Reidemeister moves in the interior of the \textit{projection disk} $B^2$. Two tangles will be \textit{equivalent} iff they have equivalent diagrams.\\
Depending on the context we will denote by $T$ the tangle or its projection.

The four endpoints of the arcs in the diagram are usually labeled $NW_T$,$NE_T$,$SE_T$, and $SW_T$ with symbols referring to the compass directions as in the Fig. \ref{fig2}.

\begin{figure}[H]
\centering
\includegraphics[width=0.2\linewidth]{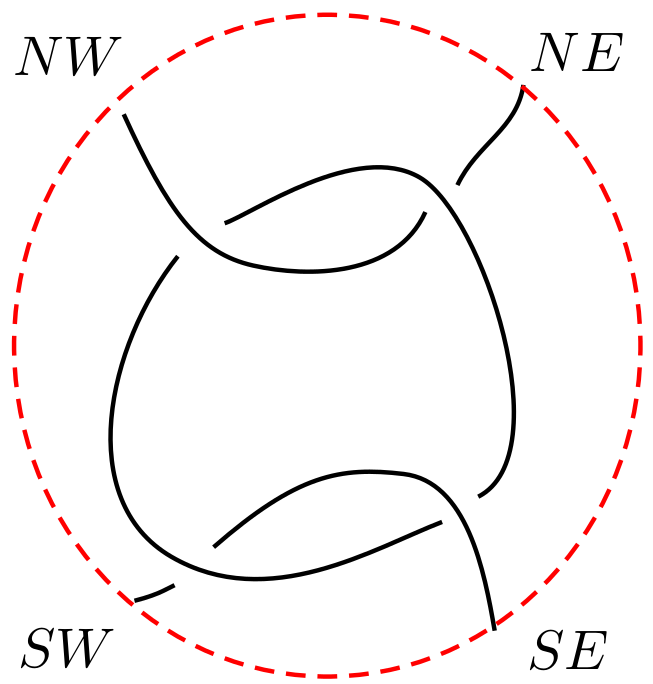}   
\caption{An alternating tangle diagram $T$.}
\label{fig2}
\end{figure}

A tangle diagram $T$ is said to be \textit{disconnected} if either there exists a simple closed curve embedded in the projection disk, called a \textit{splitting loop}, which do not meet $T$, but encircles a part of it, or there exists a simple arc properly embedded in the projection disk, called a \textit{splitting arc}, which do not meet $T$ and splits the projection disk into two disks each one containing a part of $T$. A tangle diagram is \textit{connected} if it is not disconnected.

A tangle diagram is \textit{reduced} if its number of crossings cannot be reduced by any tangle equivalence.

A tangle diagram $T$ is said to be \textit{locally knotted} if there exists a simple closed curve $C$ embedded in the interior of the projection disk, called a \textit{factorizing circle} of $T$, which meets $T$ transversally at two points and bounds a disk inside the projection disk which meets $T$ in a knotted arc.

We adopt the notations used for rational tangles by L. J ay R. Goldman and H. Kauffman in \cite{goldman1997rational} and L. H. Kauffman and S. Lambroupoulou in \cite{kauffman2004classification}. In Fig. \ref{fig4}, we recall some operations defined on tangles.

\begin{figure}[H]
\centering
\includegraphics[width=0.5\linewidth]{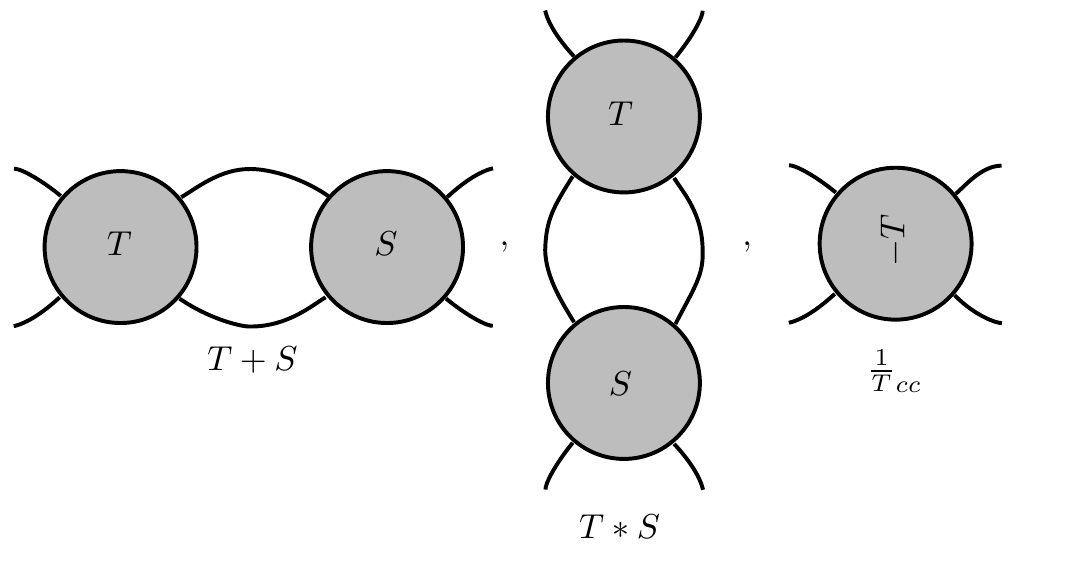}   
\caption{Some operations on tangle diagrams.}
\label{fig4}
\end{figure}

A ($-\pi$) rotation of a tangle diagram $T$ in the horizontal axis is called \textit{horizontal flip} and will be denoted by $T_h$. That is the tangle diagram obtained by rotating the ball containing $T$ in space around the horizontal axis as shown on the left in Fig. \ref{fig4bis} and then project the new tangle by the same projection function as that used to get $T$. Note that if $T$ is an alternating tangle diagram, then $T_h$ is also alternating. Note that the flip operation preseves the isotopy class of a rational tangle (Flip Theorem 1 \cite{goldman1997rational}).

\begin{figure}[H]
\centering
\includegraphics[width=0.7\linewidth]{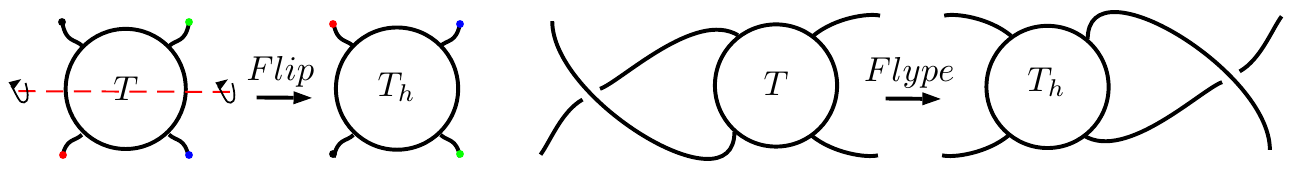}
\caption{Flip and flype operations}
\label{fig4bis}
\end{figure}

A \textit{flype} is an isotopy of tangles that is depicted on the right in the Fig. \ref{fig4bis}.
 
A tangle diagram $T$ provides two link diagrams: the \textit{Numerator} of $T$, denoted by $n(T)$, which is obtained by joining with simple arcs the two upper endpoints $(NW_T,NE_T)$ and the two lower endpoints $(SW_T,SE_T)$ of $T$, and the \textit{Denominator} of $T$, denoted by $d(T)$, which is obtained by joining with simple arcs each pair of the corresponding top and bottom endpoints $(NW_T,SW_T)$ and $(NE_T,SE_T)$ of $T$ (see Fig. \ref{fig5}). We denote $N(T)$ and $D(T)$ respectively the corresponding links. 

\begin{figure}[H]
\centering
\includegraphics[width=0.3\linewidth]{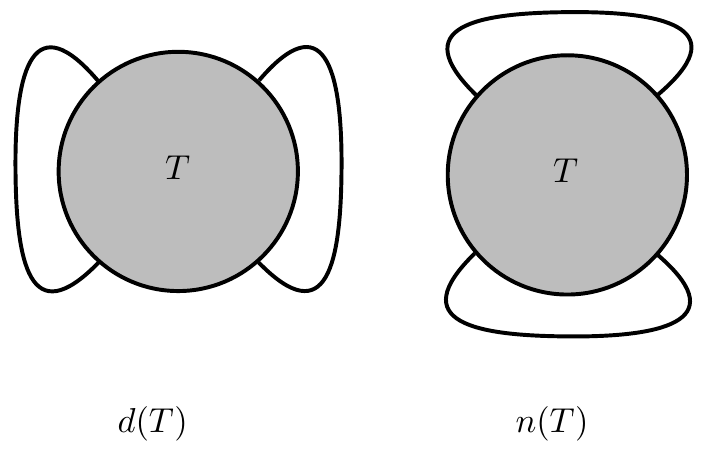}   
\caption{The denominator and the numerator of a tangle diagram $T$.}
\label{fig5}
\end{figure}

As in the case of link diagrams, one can associate to each tangle diagram $T$ a signed planar graph by choosing a checkerboard coloring of $T$. This graph will also have a dual graph. The signed planar graph of $T$ may be put inside the projection disk such that two of its vertices are evenly spaced on the boundary circle and which we call \textit{boundary vertices}. We denote it by $G_d(T)$ when its boundary vertices happen to be on the lateral sides of the boundary circle. If one boundary is on the upper arc and the other on the lower arc of the boundary circle, we will denote the graph by $G_n(T)$. Note that $G_n(T)$ and $G_d(T)$ are respectively Tait graphs of $n(T)$ and $d(T)$ (see Fig. \ref{fig6}).

\begin{figure}[H]
\centering
\includegraphics[width=0.4\linewidth]{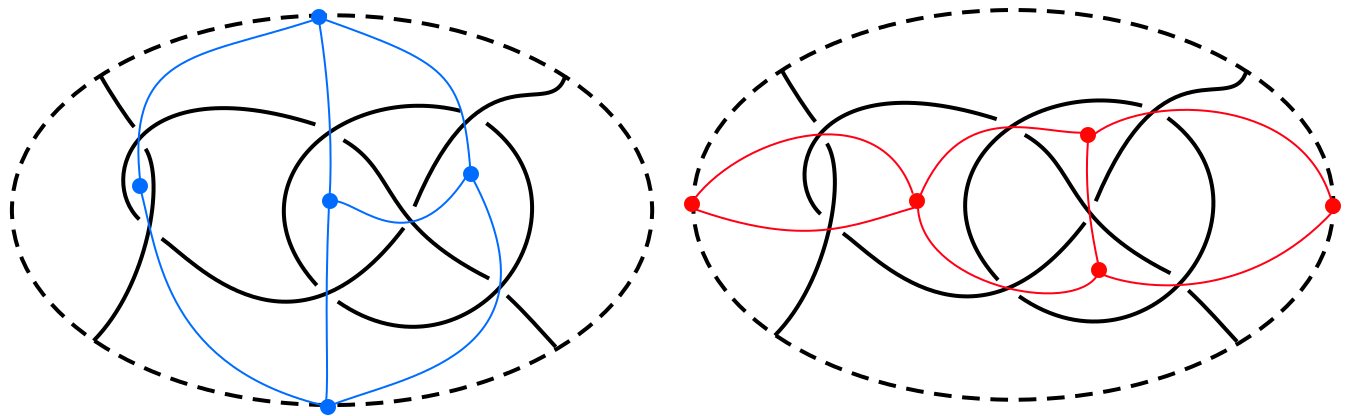}   
\caption{The graph $G_n(T)$ on the left and $G_d(T)$ on the right.}
\label{fig6}
\end{figure}

If $T$ is a connected tangle diagram, then $G_n(T)$ and $G_d(T)$ are both connected graphs (see \cite{kanenobu2003tangles}).

Denote by $u$ and $v$ the boundary vertices of $G_d(T)$. We join $u$ to $v$ by a simple arc in the exterior of the projection disk. If we coalesce $u$ and $v$ by contracting that arc, then we get the dual of $G_n(T)$ which is also a Tait graph of $n(T)$ (compare the three graphs in Fig. \ref{nfig3} and Fig. \ref{nfig4}). 

A tangle diagram $T$ is called \textit{alternating} if the ``over'' or ``under'' nature of the crossings alternates as one moves along any arc of $T$. A tangle is said to be \textit{alternating} if it admits an alternating diagram. A connected tangle diagram $T$ is alternating iff all the edges of $G_d(T)$ are of the same sign.\\

We will need the following lemma.

\begin{lem}
If $T$ is an alternating connected locally unknotted tangle diagram, then $n(T)$ or $d(T)$ is prime.
\label{lem2.2}
\end{lem}

\begin{proof}
  Let us assume that $d(T)$ is not prime. Then there exists a closed simple curve $C$ in the plane meeting $d(T)$ transversely in two points $x$ and $y$, and which factorizes $d(T)$. The curve $C$ bounds a disk $\Delta$ in the plane. It is easy to see that the only possible case for $x$ and $y$ is that they are both located inside the projection disk $B^2$.\\
Consider the connected graph $G_d(T)$ introduced above with its boundary vertices $u$ and $v$. We can assume that $G_d(T)$ meets $C$ in only one cut vertex $a$. Furthermore $\Delta$ contains a single denominator closure arc, then each connected component of $G_d(T) \setminus a$ contains a boundary vertex. 
Denote by $H_u$ and $H_v$ the connected components of $G_d(T) \setminus a$ containing respectively $u$ and $v$. Let $o$ and $o'$ be two vertices of $G_d(T)$ contained respectively in $H_u$ and $H_v$ and distinct from $a$. By connectedness of $G_d(T)$ and Theorem 6 in \cite{whitney2009non}, there exists a chain $c$ from $u$ to $v$, passing through $o$, $a$, and $o'$. When coalescing $u$ and $v$, the chain $c$ becomes a cycle $c'$ of the dual of $G_n(T)$ containing $u$, $v$, $o$, $a$, and $o'$. Note that any cycle of $G_d(T)$ is also a cycle of the dual of $G_n(T)$ by coalescing $u$ and $v$. Then $o$ and $o'$ are contained in the cycle $c'$ of the dual of $G_n(T)$. If $o$ and $o'$ are both in the same non-separable component of $G_d(T)$, it is easy to find a cycle of the dual of $G_n(T)$ containing them both. So, this last graph is non-separable by Theorem 7 in \cite{whitney2009non}. Hence, the graph $G_n(T)$ is also non-separable . Since, as stated in \citep{thistlethwaite1987spanning}, non-separable graphs correspond to prime link diagrams, then, $n(T)$ is prime.

Since $n(T) = d(\frac{1}{T}_{cc})$, then the above argument shows that when $n(T)$ is composite, $d(T)$ is prime.
\end{proof}
Let $T$ be an alternating connected tangle diagram. Consider the arc of $T$ which have $NW_T$ as an endpoint. Suppose that when we move along that arc starting at $NW_T$ we pass below at the first encountered crossing. Then all edges of $G_n(T)$ will be positive, all edges of $G_d(T)$ will be negative, the arc of $T$ which ends at the point $SE_T$ will also pass below at the last encountered crossing before reaching $SE_T$ and the arc of $T$ which starts at $NE_T$ will pass over at the first encountered crossing. It is easy to see that the arcs of $T$ coming from diametrically opposite endpoints both pass over or below at the first encountered crossing. That remark enables us to distinguish two types of alternating connected tangle diagrams which we call \textit{type 1} tangles and \textit{type 2} tangles as shown in Fig. \ref{fig7}. Throughout the rest of this paper, all the considered alternating tangle diagrams will be assumed to be of type 1 unless otherwise stated. 

\begin{figure}[H]
\centering
\includegraphics[width=0.4\linewidth]{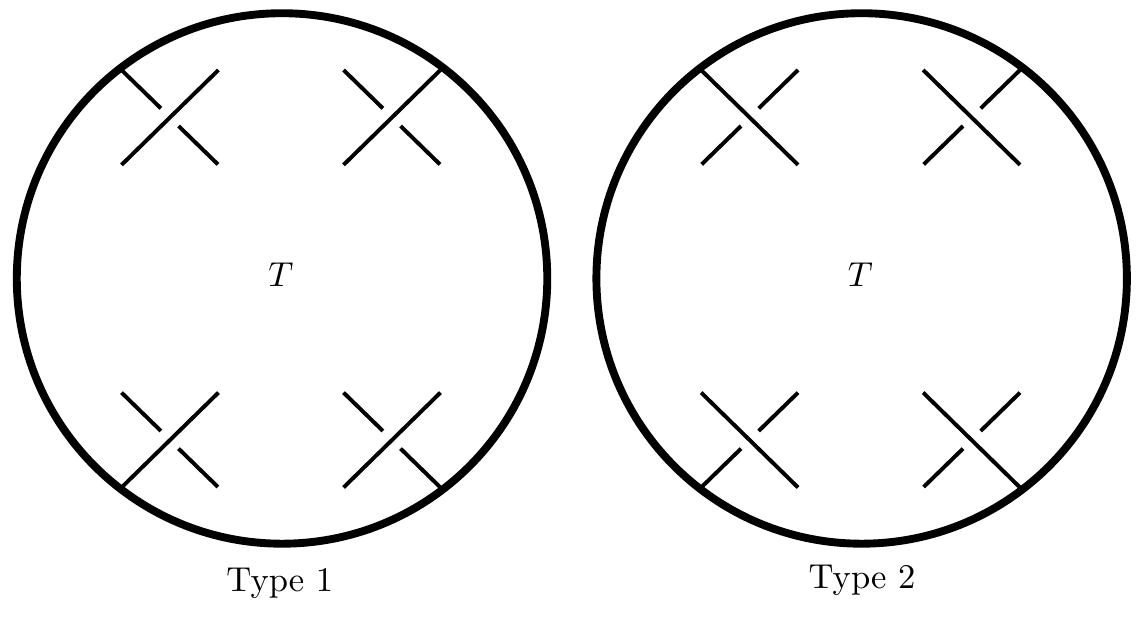}   
\caption{Type 1 and Type 2 alternating connected tangle diagrams.}
\label{fig7}
\end{figure}

Let $T$ be an alternating tangle diagram in $B^2$. Then $n(T)$ and $d(T)$ are both alternating diagrams. If $n(T)$ and $d(T)$ are also connected and reduced, then $T$ is said to be \textit{strongly alternating}.

\subsection{Rational tangles}

A \textit{rational tangle} $t$ is a tangle in $B^3$ such that the pair $(B^3, t )$ is homeomorphic to $(B^2 \times [0,1] , \left\lbrace x,y \right\rbrace \times [0,1]) $, where $x$ and $y$ are points in the interior of $B^2$. The elementary rational tangle diagrams $0$, $\pm1$, $\infty$ are shown in Fig. \ref{fig8}.
\begin{figure}[H]
\centering
\includegraphics[width=0.15\linewidth]{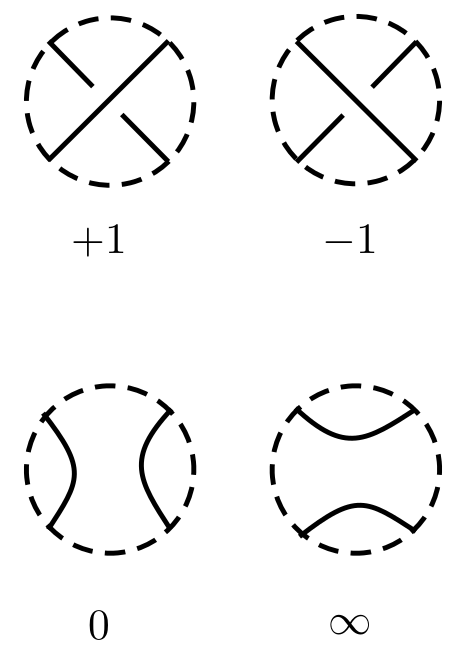}  
\caption{Elementary rational tangles.}
\label{fig8}
\end{figure}
The sum of $n$ copies of the tangle diagram $1$ or of $n$ copies of the tangle $-1$ are respectively the \textit{integral tangle diagrams} denoted also by $n$ and $-n$.
If $t$ is a rational tangle diagram then $\frac{1}{t}_c$ and $\frac{1}{t}_{cc}$ are equivalent and both represent the \textit{inversion} of $t$ denoted by $\frac{1}{t}$.

 Let $t$ be a rational tangle diagram and $p,q \in \mathbb{Z}$, we have the following equivalences:
$$  p+t+q = t+p+q \text{  ,  } \frac{1}{p} * t * \frac{1}{q} = t * \frac{1}{p+q}.$$
$$ t * \frac{1}{p} = \frac{1}{p+\frac{1}{t}} \text{  ,  } \frac{1}{p} * t = \frac{1}{\frac{1}{t}+p}. $$

Using the above notations and equivalences one can naturally associate to any continued fraction
$$\displaystyle a_1 + \frac{1}{ a_2  + \frac{1}{  \ddots +\frac{1}{ a_{n-1}  + \frac{1}{ a_n } }}},\,\ a_i\in\mathbb{Z},$$ a tangle diagram as shown in Fig. \ref{fig9} denoted by $\left[ a_1,...,a_{n} \right]$. 

\begin{figure}[!htbp]
\centering
\includegraphics[width=0.8\linewidth]{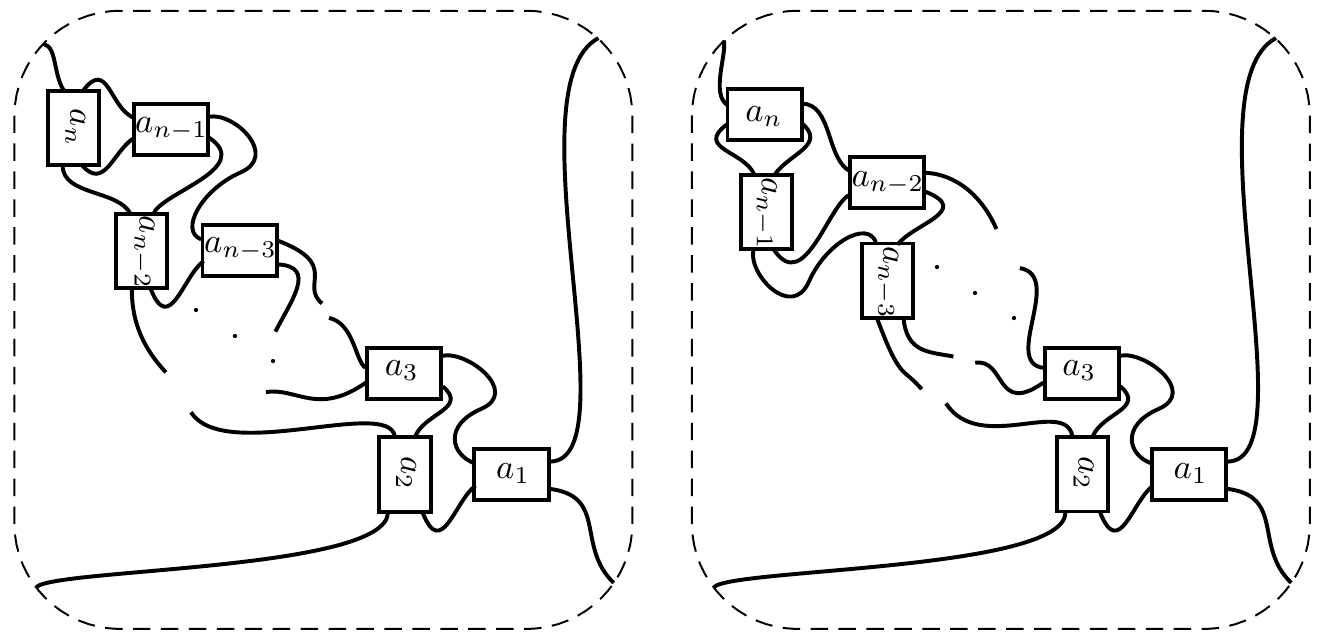}  
\caption{The standard rational tangle diagram $\left[ a_1,...,a_{n} \right]$ according to $n$ is even (left) or odd (right).}
\label{fig9}
\end{figure}

Conversely, it is known that for any rational tangle $t$, there exists an integer $n \geq 1$ and integers $ a_1 \in \mathbb{Z}$, $a_2, ... , a_n \in \mathbb{Z}\setminus\lbrace 0 \rbrace,$ all of the same sign, such that  $t= \left[ a_1,...,a_{n} \right]$. Then $t$ corresponds to a continued fraction and then to a rational number called the fraction of the tangle.

J. H. Conway showed in \cite{conway1970enumeration} that two rational tangles are equivalent if and only if they have the same fraction. Then any rational tangle $t$ can be represented by a continued fraction $\left[ a_1  , ... ,  a_n \right] = \frac{a}{b}$ where $a$ and $b$ are two coprime integers. 

The \textit{standard diagram} of a rational tangle $t$ will be the alternating connected reduced locally unknotted diagram naturally associated to the continued fraction of $t$ described above. In what follows a rational tangle diagram will mean the standard one.

\subsection{Rational extensions}
Let $c$ be a crossing of some link diagram. It can be considered as a tangle with marked end points. By using Conway's notation for rational tangles, we assign to $c$ the number $\varepsilon(c) \in \left\lbrace -1,1 \right\rbrace$ according to whether the overstrand has negative or positive slope as depicted in the figure below. 

\begin{figure}[H]
\centering
\includegraphics[width=0.3\linewidth]{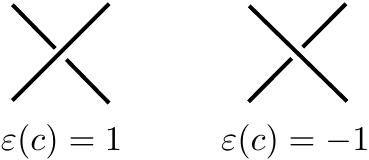}  
\end{figure}

We recall the technic of rational extension used by A. Champanerkar and I. Kofman in \cite{champanerkar2009twisting}. We say that a rational tangle $t=\left[ a_1, ... , a_n \right] $ \textit{extends} the crossing $c$ if $t$ contains $c$ and for all $i$, $a_i\cdot \varepsilon(c) \geq 1$.
That means that all crossings of $t$ have the same Conway sign as that of $c$. 
One can always replace a crossing $c$ in some link diagram $D$ with a rational tangle $t$ which extends it to get a new link diagram which we will denote by $D^{c \leftarrow t}$. This diagrammatic operation is called a \textit{rational extension} of $c$ with $t$. We will say that $c$ is \textit{the extended crossing}, and $t$ is an \textit{extension tangle} of $c$. The link diagram $D^{c\leftarrow t}$ will be called a rational extension of $D$ at $c$.

A rational extension with fraction $\frac{13}{4}=3+\frac{1}{4}$ of the diagram of the Hopf link is depicted in the Fig. \ref{fig10}.

\begin{figure}[H]
\centering
\includegraphics[width=0.8\linewidth]{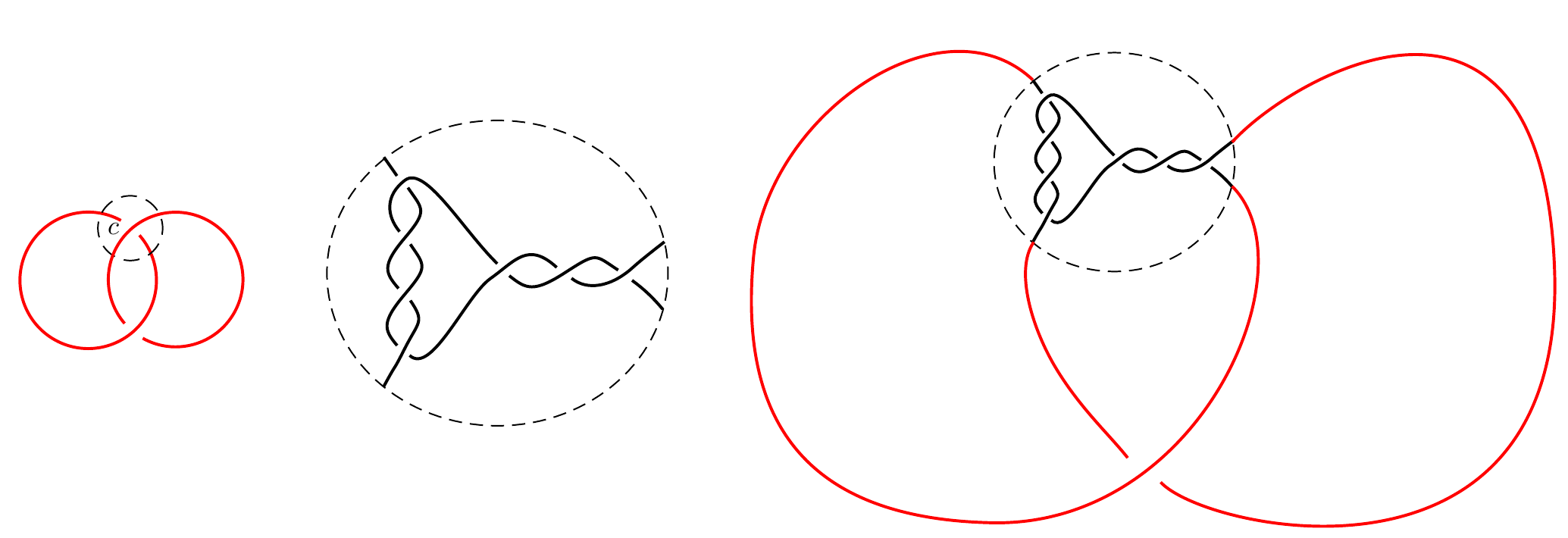}  
\caption{From the left: Hopf link, rational tangle $[3,4]=\frac{13}{4}$ and rational extension $D^{c\leftarrow \frac{13}{4}}$. }
\label{fig10}
\end{figure}

\subsection{Almost alternating links}
A link diagram $D$ is said to be \textit{almost alternating} if one crossing change makes it alternating. A crossing whose change yields an alternating diagram is called a \textit{dealternator}. A link $L$ is said to be almost alternating if it is not alternating and has an almost alternating diagram. Recall that the category of almost alternating links was first introduced by Adams, Brock, Bugbee, Comar, Huston, and Jose in \cite{adams1992almost}.

Note that if a link diagram is an almost alternating diagram with one dealternator then the edges of its Tait graph are all of the same sign except that associated to the dealternator which will be of opposite sign. Throughout the rest of the paper, we will consider only the Tait graphs for which the edge corresponding to the dealternator is negative.

Let $D$ be an almost alternating link diagram with dealternator $c$. We will consider $c$ as the rational tangle $-1$ (rotate $D$ if necessary). If both smoothings $D^c_0$ and $D^c_\infty$ of $D$ at $c$ are reduced alternating link diagrams, then $D$ is said to be \textit{dealternator reduced}. If those both smoothings are connected alternating link diagrams, then $D$ is said to be \textit{dealternator connected}. A typical example of a dealternator connected reduced link diagram that we will need in the following is the link diagram  $n(-1+T)$ where $T$ is a strongly alternating tangle diagram.

\begin{rmq}
  Let $L$ be an almost alternating link. Let $D$ be an almost
  alternating diagram of $L$ which has the smallest crossing number
  among all almost alternating diagrams of $L$. Then $D$ has only one
  dealternator $c$ and it is a dealternator connected reduced diagram as
  showed in the proof of Corollary 4.5 in
  \cite{adams1992almost}. It is easy to show that the diagram $D$ is
  equivalent to the numerator $n(-1+T)$ where $T$ is a strongly
  alternating tangle diagram. Hence $L$ is equivalent to $N(-1+T)$.
\end{rmq}
  On the other hand, in the following, we will need to study if a link $L=N(-1+T)$ is quasi-alternating. To do this, we will use Remark \ref{rk3} below.
\begin{rmq}\label{rk3}
Consider a link $L=N(-1+T)$. If $T$ is locally knotted, then $L$ has at least one alternating factor. Let $C$ be a factorizing circle of $T$. Denote by $a$ and $b$ the intersection points of $C$ and $T$. Remove the disk bounded by $C$ and containing the alternating factor. Then, replace it with a disk containing one simple arc joining $a$ to $b$. We can repeat this operation until all alternating factors will be removed. Then we get a link $L^{'}=N(-1+T^{'})$ where $T^{'}$ is a locally unknotted tangle diagram. Those operations preserve the property of being quasi-alternating. Conversely, the link $L$ is the connected sum of $L^{'}$ with some alternating factors. Furthermore, the connected sum of any quasi-alternating links is quasi-alternating \cite{ozsvath2005heegaard}. Finally, the link $L$ is quasi alternating if and only if $L^{'}$ is quasi alternating. So we can restrict to links $L=N(-1+T)$ where $T$ is locally unknotted.
\end{rmq}

\section{Dealternator extensions\label{sect5}}
Let $D$ be an almost alternating diagram with dealternator $c$. The determinant of $D$ is given in term of the determinants of $D^{c}_0$ and $D^{c}_\infty$ as stated in the following proposition.

\begin{prop}
Let $D$ be an almost alternating diagram with dealternator $c$. Then $$ \det(D) = \left| \det \left( D^{c}_{0} \right) - \det \left( D^{c}_{\infty} \right) \right|.$$
\label{prop3.1}
\end{prop}

\begin{proof}
Let $G_D$ be the Tait graph of $D$ such that the unique negative edge is that one corresponding to $c$. Call that edge $e_c$. The set $\mathcal{T}(G_D)$ of spanning trees of $G_D$ admits a partition $ \mathcal{T}^{c}_{0}(G_D) \uplus \mathcal{T}^{c}_{\infty}(G_D)$ where: $$ \mathcal{T}^{c}_{\infty}(G_D) = \lbrace T \in \mathcal{T}(G_D) / e_c \in T \rbrace \text{ , and } \mathcal{T}^{c}_{0}(G_D) = \lbrace T \in \mathcal{T}(G_D) / e_c \notin T \rbrace.$$ A tree in $\mathcal{T}^{c}_{0}(G_D)$ can be seen as a spanning tree of $G_{D^{c}_{0}}$. And if we contract the edge $e_c$ in a tree in $\mathcal{T}^{c}_{\infty}(G_D)$ we get a spanning tree of $G_{D^{c}_{\infty}}.$ Those correspondences are actually one-to-one following \cite{champanerkar2009twisting}. 
Let $T \in \mathcal{T}(G_D)$, then $e(T)=e_+(T) + e_-(T)=v(G_D)-1$. If $T$ is in $\mathcal{T}^{c}_{\infty}(G_D)$, then $e(T)= e_+(T) + 1=v(G_D)-1$, which gives $e_+(T)=v(G_D)-2$. If $T$ is in $\mathcal{T}^{c}_{0}(G_D)$, then $e(T)= e_+(T) =v(G_D)-1$. This way we have 
\begin{align*}
\displaystyle \sum_k (-1)^k s_k(D) =& (-1)^{v(G_D)-1} s_{v(G_D)-1}(D) + (-1)^{v(G_D)-2} s_{v(G_D)-2}(D) \\
=& (-1)^{v(G_D)-1} \# \mathcal{T}^{c}_{0}(G_D)  + (-1)^{v(G_D)-2} \# \mathcal{T}^{c}_{\infty}(G_D) \\
=& (-1)^{v(G_D)-1} \# \mathcal{T}(G_{D^{c}_{0}})  + (-1)^{v(G_D)-2} \# \mathcal{T}(G_{D^{c}_{\infty}}) \\
=& (-1)^{v(G_D)-1} \det(D^{c}_{0}) + (-1)^{v(G_D)-2} \det(D^{c}_{\infty}) \\
=& (-1)^{v(G_D)-1} \left( \det(D^{c}_{0}) - \det(D^{c}_{\infty}) \right)
\end{align*}
The result then follows using Theorem \ref{th2.1}.
\end{proof}

\begin{rmq}
A link with zero determinant is not quasi-alternating. This implies that if an almost alternating diagram $D$ is representing a quasi-alternating link, then the determinants of $D^{c}_0$ and $D^{c}_\infty$ must differ if $c$ is the dealternator of $D$. When $\mathcal{L}(D)$ is also almost alternating, this diffrence has a lower bound (see Corollary 1.2, \cite{lidman2017quasi}).
\end{rmq}

\begin{cor}
Let $D$ be an almost alternating link diagram with dealternator $c$. Suppose that $\mathcal{L}(D)$ is quasi-alternating, then $$ \det \left( D^{c}_{0} \right) \neq \det \left( D^{c}_{\infty} \right) .$$ If $\mathcal{L}(D)$ is also almost alternating, then $\left| \det \left( D^{c}_{0} \right) - \det \left( D^{c}_{\infty} \right) \right| \geq 8$.
\label{cor1}
\end{cor}

\begin{cor}
An almost alternating diagram is never quasi-alternating at its dealternator.
\label{cor2}
\end{cor}

\begin{prop}
Let $T$ be an alternating tangle diagram and let $\dfrac{\alpha}{\beta} > 0$ be a rational number. We have the following 
\begin{enumerate}
\item $\det(n(\frac{\alpha}{\beta}+T)) = \beta\det(n(T))+\alpha\det(d(T)).$
\item $\det(n(-\frac{\alpha}{\beta}+T)) = \left| \beta\det(n(T))- \alpha\det(d(T)) \right|.$
\end{enumerate}
\label{prop3.2} 
\end{prop}

\begin{rmq}
If $t$ and $T$ are respectively a rational and an alternating tangle diagrams, the link diagram $n(t+T)$ is equivalent, up to mirror image, to the link diagram $n(\frac{1}{t} + \frac{1}{T}_c)$ as shown in the figure below. Since the tangle diagram $\frac{1}{T}_c$ is also alternating, one can only restrict to the case where $0<\left| t \right| < 1$ when considering the numerator closure of $t$ summed with an alternating tangle diagram.
\begin{figure}[H]
\centering
\includegraphics[width=0.8\linewidth]{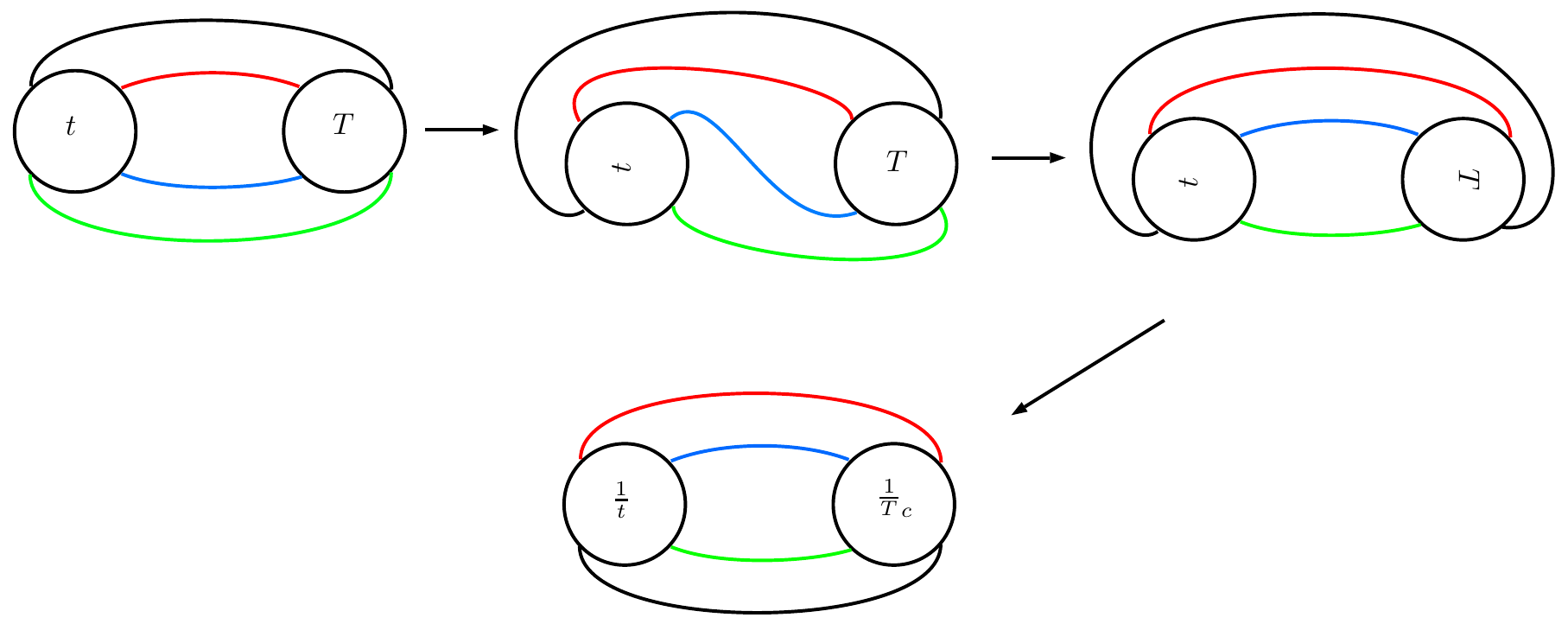}  
\end{figure}
\label{rmq1}
\end{rmq}
\begin{proof}[Proof of Proposition \ref{prop3.2}] \underline{First part of the proposition}: by Remark \ref{rmq1}, we can assume that $0 < \frac{\alpha}{\beta} < 1$ and then we write $\frac{\alpha}{\beta} = \left[ 0,a_1,...,a_k \right]$ where $k$ is a positive integer such that $a_i \geq 1$ for each $i$. We will use induction on the number $k$ of integer tangles which consist the rational tangle diagram corresponding to the fraction $\frac{\alpha}{\beta}= \left[ 0,a_1,...,a_k \right]$. For $k=1$, we have $\frac{\alpha}{\beta} = \frac{1}{p}$ for some integer $p \geq 1$. In this case we have $n(\frac{\alpha}{\beta}+T)=n(\frac{1}{p}+T)$ which is an alternating link diagram which we denote by $D_{p}$. Let us label by $c_1,...,c_p$ the crossings of the vertical tangle $\frac{1}{p}$ respectively from the top. It is clear that $(D_j)^{c_s}_{0}=D_{j-1}$ and $(D_j)^{c_s}_{\infty}=n(T)$ for each $j$, $1 \leq j \leq p$, and for each $s$, $1 \leq s \leq j$. Since the link diagram $D_j$ is quasi-alternating at the crossing $c_j$, we have 

\begin{align*}
\det(n(\frac{1}{p}+T))  &= \det(D_p) \\
&= \det(D_{p-1}) + \det(n(T)) \\
& \vdots \\
&= p\det(n(T)) + \det(d(T)).
\end{align*}

Suppose the result holds until some rank $k$. Put $\frac{\alpha}{\beta}= \left[ 0,a_1,...,a_k \right]$ and $\frac{\gamma}{\delta}=\left[ 0,q,a_1,...,a_k \right] $ for some integer $q \geq 1$. It is clear that $\left[ 0,q,a_1,...,a_k \right] = \left[a_1,...,a_k \right] * \frac{1}{q}$. Thus, we have $\frac{\gamma}{\delta}= \frac{\beta}{\alpha} * \frac{1}{q}=\frac{1}{q+\frac{\alpha}{\beta}}=\frac{\beta}{q\beta + \alpha}$.
Denote by $\mathcal{D}_j$ the link diagram $n(\frac{\beta}{\alpha} * \frac{1}{j} + T)$ for each $j$, $0 \leq j \leq q$. Also denote by $c_1,...,c_j$ the crossings of the vertical tangle $\frac{1}{j}$ respectively from the top. It is clear that $(\mathcal{D}_j)^{c_j}_{0} = \mathcal{D}_{j-1}$, and $(\mathcal{D}_j)^{c_j}_{\infty} = n(\frac{\beta}{\alpha}) \# n(T)$. Now since the link diagram $\mathcal{D}_j$ is quasi-alternating at the crossing $c_j$ for every $1 \leq j \leq q$, we have
\begin{align*}
\det(\frac{\gamma}{\delta} + T) &= \det(\mathcal{D}_q) \\
&= \det(\mathcal{D}_{q-1}) + \beta \det(n(T)) \\
& \vdots \\
&= \det(n(\frac{\beta}{\alpha} + T)) + q\beta \det(n(T)).
\end{align*}
By Remark \ref{rmq1}, we have $\det(n(\frac{\beta}{\alpha} + T)) = \det(n(\frac{\alpha}{\beta} + \frac{1}{T}_c))$. Since $\frac{\alpha}{\beta}=\left[ 0,a_1,...,a_k \right]$ consists of $k$ integer tangles, then by the induction hypothesis we have 
\begin{align*}
\det(n(\frac{\beta}{\alpha} + T)) &= \det(n(\frac{\alpha}{\beta} + \frac{1}{T}_c)) \\
&= \beta \det(n(\frac{1}{T}_c)) + \alpha \det(d(\frac{1}{T}_c)) \\
&= \beta \det(d(T)) + \alpha \det(n(T)).
\end{align*}
This way we have 
\begin{align*}
\det(\frac{\gamma}{\delta} + T) &= \det(n(\frac{\beta}{\alpha} + T)) + q\beta \det(n(T)) \\
&= \beta \det(d(T)) + \alpha \det(n(T)) + q\beta \det(n(T)) \\
&= (\alpha+q\beta)\det(n(T)) + \beta \det(d(T)) \\
&= \delta\det(n(T)) + \gamma\det(d(T)).
\end{align*}
This completes the induction argument.\\
\underline{Second part of the proposition}: Let $D$ denote the almost alternating link diagram $n(-1+\frac{\beta-\alpha}{\beta}+T)$ and let $c$ denote its dealternator. The link diagram $D$ is equivalent to $n(\frac{-\alpha}{\beta}+T)$ by the rational tangle equivalence $-1+\frac{\beta-\alpha}{\beta}=\frac{-\alpha}{\beta}$. The link diagrams $D^{c}_{0}$ and $D^{c}_{\infty}$ are respectively equivalent to $n(\frac{\beta-\alpha}{\beta}+T)$ and $d(\frac{\beta-\alpha}{\beta}) \# d(T)$. We have $\det(D^{c}_{0})= (\beta-\alpha)\det(d(T))+\beta\det(n(T))$, and $\det(D^{c}_{\infty})=\beta\det(d(T))$. By Proposition \ref{prop3.1} we have 
\begin{align*}
\det(D) &= \det(n(-\frac{\alpha}{\beta}+T)) \\
&=\left| \det(n(\frac{\beta-\alpha}{\beta}+T))-\det(d(\frac{\beta-\alpha}{\beta}) \# d(T)) \right| \\
&=\left| (\beta-\alpha)\det(d(T))+\beta\det(n(T))- \beta\det(d(T)) \right| \\
&=\left| \beta\det(n(T))- \alpha\det(d(T)) \right|.
\end{align*}
\end{proof}

\begin{cor}
Let $D$ be an almost alternating link diagram with dealternator $c$. Then $$ \det \left( D^{c \leftarrow\frac{1}{-n}} \right) = \left| n \times \det \left( D^{c}_{0} \right) - \det \left( D^{c}_{\infty} \right) \right|, \det \left( D^{c \leftarrow-n} \right) = \left| n \times \det \left( D^{c}_{\infty} \right) - \det \left( D^{c}_{0} \right) \right|.$$
\label{cor3}
\end{cor}

\begin{proof}
$D$ can be represented as $n(-1+T)$ where $T$ is an alternating tangle diagram (not necessarily strongly alternating in this case). The link diagrams $D^{c}_{0}$ and $D^{c}_{\infty}$ would be equivalent to $n(T)$ and $d(T)$ respectively. The result follows by using Proposition \ref{prop3.2}.
\end{proof}

Let us take two strongly alternating tangle diagrams $T$ and $S$ of different types. The link diagram $n(T+S)$ is said to be \textit{semi-alternating}. A link is \textit{semi-alternating} if it admits a semi-alternating diagram. Semi-alternating links are non-split as shown in Proposition 6 of \citep{lickorish1988some}. Semi-alternating links represent a special case of \textit{adequate links}, which we do not define in this paper. Adequate links are non-quasi-alternating links since they have thick Khovanov homology (Proposition 7 in \cite{khovanov2003patterns}). By using the notion of semi-alternating links, we get the following proposition.

\begin{prop}
Let $T$ be a strongly alternating tangle diagram and $\dfrac{a}{b}$ be a rational number, $0<\dfrac{a}{b}<1$. If the link $N(-\dfrac{a}{b} + T)$ is quasi-alternating, then $\dfrac{\det(n(T))}{\det(d(T))} > \dfrac{a}{b}$. 
\label{prop3.3} 
\end{prop}
\begin{proof}
Since $N(\frac{-a}{b}+T)$ is quasi-alternating, then its determinant is not zero. Hence by  Proposition \ref{prop3.2} we have $\frac{\det(n(T))}{\det(d(T))} \neq \frac{a}{b}$. 

Suppose that $\frac{\det(n(T))}{\det(d(T))} < \frac{a}{b}$. Let $D'$ denote the link diagram $n(-1+\frac{-a}{b}+T)$ and let $c'$ denote the leftmost crossing of $D'$. The link diagram $D'$ is equivalent to $n(-2+\frac{b-a}{b}+T)$ because of the rational tangle equivalence $-2+\frac{b-a}{b} = -1+\frac{-a}{b}$. By Proposition \ref{prop3.2} we have 
\begin{align*}
\det(D')&=\det(n(-2+\frac{b-a}{b}+T)) \\
&=\left| 2\det(d(\frac{b-a}{b}+T)) - \det(n(\frac{b-a}{b}+T)) \right| \\
&=(b+a)\det(d(T))-b\det(n(T)).
\end{align*}
On the other hand, the link diagrams $D'^{c'}_{0}$ and $D'^{c'}_{\infty}$ are respectively equivalent to $n(\frac{-a}{b}+T)$ and $d(\frac{-a}{b}) \# d(T)$. This implies that $\mathcal{L}( D'^{c'}_{\infty})$ is quasi-alternating because it is a connected sum of alternating links. The link $\mathcal{L}( D'^{c'}_{0})$ is quasi-alternating by assumption. On the other hand, we have the following
\begin{align*}
\det(D'^{c'}_{0})+\det(D'^{c'}_{\infty})&=a\det(d(T))-b\det(n(T))+b\det(d(T)) \\
&=(b+a)\det(d(T))-b\det(n(T)) \\
&=\det(D').
\end{align*}
This implies that the link diagram $D'$ is quasi-alternating at the crossing $c'$. Hence, by Theorem \ref{th1.1} , the link diagram $D'^{c'\leftarrow\frac{-1}{2}}$, which is equivalent to the semi-alternating diagram $n((\frac{-1}{2}+\frac{-a}{b})+T)$, is quasi-alternating. This is absurd since semi-alternating links are non-quasi-alternating.
Consequently, the inequality $\frac{\det(n(T))}{\det(d(T))} > \frac{a}{b}$ is true. 
\end{proof}

\begin{empl}
  Let $L$ be the link $N(-\left[ 0,1,10,1,6 \right]+T)$ depicted in Fig. \ref{fig14bis}. We have
  $$\frac{\det(n(T))}{\det(d(T))} = \frac{65}{71}<\frac{76}{83}=\left[ 0,1,10,1,6 \right].$$ Then by Proposition \ref{prop3.3} $L$ is not quasi-alternating.
\begin{figure}[H]
\centering
\includegraphics[width=0.5\linewidth]{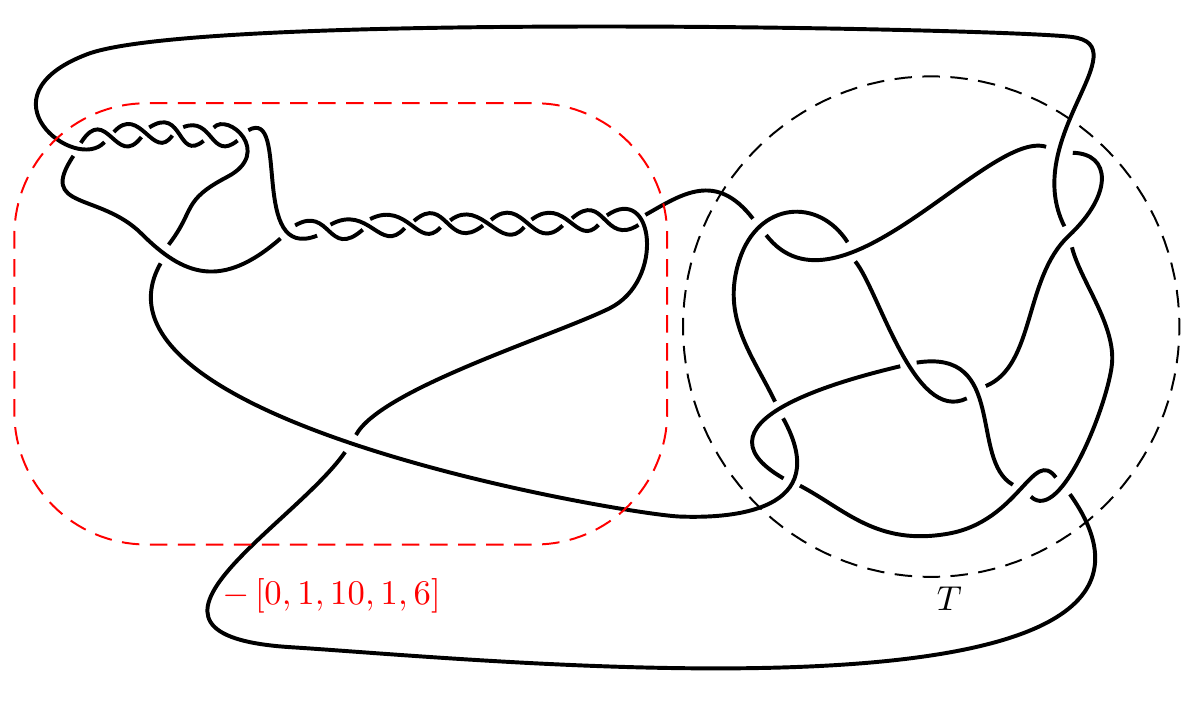}
\caption{A non-quasi-alternating link.\label{fig14bis}}
\end{figure}
\end{empl}

\begin{rmq}
Let $T$ be a strongly alternating tangle diagram and $t$ be a rational tangle diagram, $0<t<1$. The Proposition \ref{prop3.3} provides an obstruction criterion for quasi-alternateness of the link $L=N(-t+T)$. However it is not a sufficient condition. Indeed, if $t=\frac{1}{2}$ and $T=\frac{1}{3}+\frac{1}{3}$, although the condition holds, the link $N(-t+T)$, which is equivalent to the Montesinos link $M(-1;2,3,3)$, is not quasi-alternating by Theorem \ref{th4.1}.  
\label{rmq6}
\end{rmq}

\begin{cor}
Let $T$ be a strongly-alternating tangle diagram.
\begin{enumerate}
\item If $N(\frac{-1}{2}+T)$ is quasi-alternating, then the link diagram $n(\frac{-1}{k}+T)$ is quasi-alternating at every crossing of the vertical tangle $\dfrac{-1}{k}$ for every integer $k \geq 3$.
\item If $N(-2+T)$ is quasi-alternating, then the link diagram $n(-k+T)$ is quasi-alternating at every crossing of the integer tangle $-k$ for every integer $k \geq 3$.
\end{enumerate}
\label{cor5}
\end{cor}
\begin{proof}
\begin{enumerate}
\item If $N(\frac{-1}{2}+T)$ is quasi-alternating, then by Proposition \ref{prop3.3} we have $\frac{\det(n(T)}{\det(d(T))} > \frac{1}{2}$, which is equivalent to $2\det(n(T))-\det(d(T)) > 0$. This last condition is enough to show that the link diagram $n(\frac{-1}{3}+T)$ is quasi-alternating at each crossing of the vertical tangle $\frac{-1}{3}$. The result then follows by Theaorem \ref{th1.1}.
\item The result follows by an analogous argument.
\end{enumerate}
\end{proof}
Now we are ready to prove Theorem \ref{mytheo1}. We will start by giving an expanded version.
\begin{theo}
  Let $D$ be an almost alternating diagram representing a quasi-alternating link. Denote $c$ its dealternator. We have the following properties.
  \begin{enumerate}
  \item If $\det \left( D^{c}_{0} \right) >\det \left( D^{c}_{\infty} \right)$, then the rational extension with fraction $-\dfrac{1}{2}$ 
of $D$ at $c$ yields a quasi-alternating diagram.
\item If $\det \left( D^{c}_{0} \right) <\det \left( D^{c}_{\infty} \right)$, then the rational extension with fraction $-2$  of $D$ at $c$ yields a quasi-alternating diagram.
\end{enumerate}
\label{th3.2}
\end{theo}

\begin{proof}
First, sufficient and necessary conditions will be exhibited for $D^{c \leftarrow-\frac{1}{2}}$ to be quasi-alternating at $c$. 
$D^{c \leftarrow-\frac{1}{2}}$ is quasi-alternating at $c$ if and only if 
\begin{align}
 \label{cond1}
 \mathcal{L}\left(\left(D^{c \leftarrow-\frac{1}{2}} \right)^{c}_{0}\right),\mathcal{L}\left(\left(D^{c \leftarrow-\frac{1}{2}} \right)^{c}_{\infty}\right) \in \mathcal{Q}, \text{ and }
\end{align}
\begin{align}
 \label{cond2}
 \det\left(D^{c \leftarrow-\frac{1}{2}} \right) = \det\left(\left(D^{c \leftarrow-\frac{1}{2}} \right)^{c}_{0}\right) + \det\left(\left(D^{c \leftarrow-\frac{1}{2}} \right)^{c}_{\infty}\right).
\end{align}
Since $\mathcal{L}\left(\left(D^{c \leftarrow-\frac{1}{2}} \right)^{c}_{0}\right) = \mathcal{L}(D^{c}_{0})$ is alternating and $\mathcal{L}\left(\left(D^{c \leftarrow-\frac{1}{2}}\right)^{c}_{\infty}\right) = \mathcal{L}(D)$ is quasi-alternating by assumption, then condition (\ref{cond1}) is always satisfied. On the other hand, condition (\ref{cond2}) is , by Proposition \ref{prop3.2} and Corollary \ref{cor1}, equivalent to $ \det(D^{c}_{0}) > \det(D^{c}_{\infty})$. Finally we get that, under the assumptions of the theorem, $D^{c \leftarrow-\frac{1}{2}}$ is quasi-alternating at $c$ if and only if $\det(D^{c}_{0}) > \det( D^{c}_{\infty})$.
By an analogical reasoning we obtain that $D^{c\leftarrow-2}$ is quasi-alternating at $c$ if and only if $ \det(D^{c}_{0}) < \det(D^{c}_{\infty})$.
Following Corollary \ref{cor1}, either $D^{c \leftarrow-\frac{1}{2}}$ or $D^{c\leftarrow-2}$ is quasi-alternating at $c$.
\end{proof}

\begin{empl}
We apply Theorem \ref{th3.2} to the link diagram on the left of Fig. \ref{exmpl11}. It is an almost alternating diagram which is quasi-alternating at the marked crossing. It represents the tabulated link $L7n2$ in \cite{KnotInfo}. Then we get the quasi-alternating diagram on the right of Fig. \ref{exmpl11}. This has one additional component and represents the tabulated link $L9n28$ in \cite{KnotInfo}. 
\begin{figure}[H]
\centering
\includegraphics[width=0.8\linewidth]{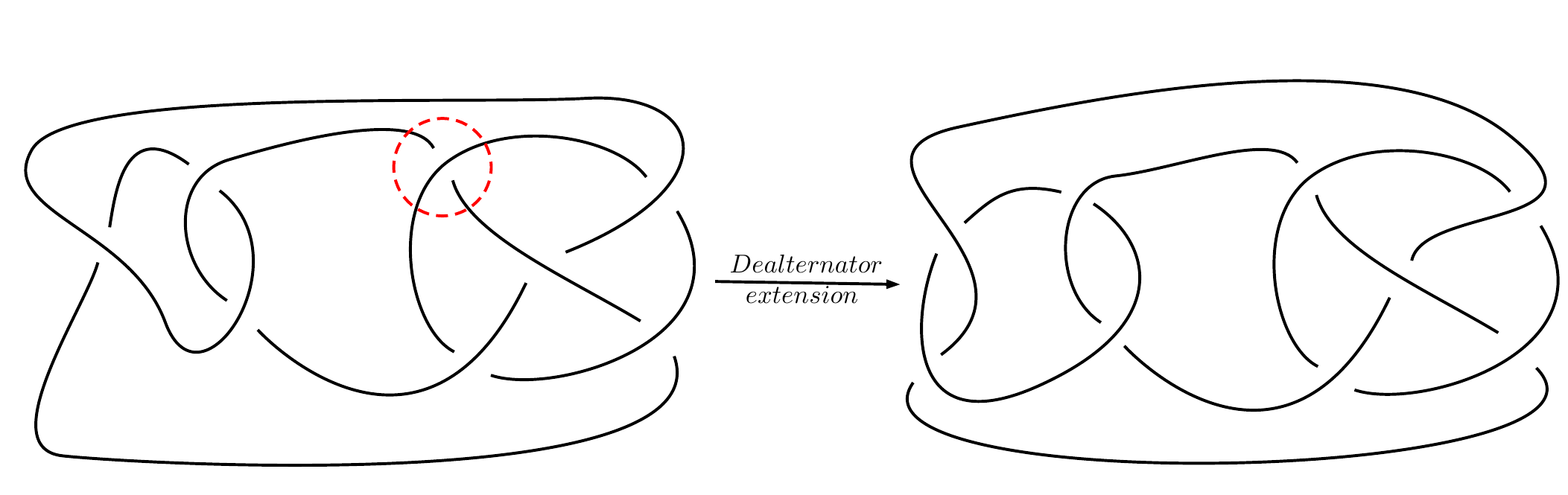} 
\caption{A dealternator extension of a diagram of the link $L7n2$ that provides the quasi-alternating link $L9n28$.}
\label{exmpl11} 
\end{figure}
\end{empl}

\section{Applications}
\subsection{Generating non-alternating and quasi-alternating Montesinos links}

Let $t_i \neq 0 , \pm 1$, for $i \in [\![1,n]\!]$, be rational numbers, and let $e$ be an integer. A \textit{Montesinos link} is defined as $ M(e;t_1,...,t_n) := N(e+ \frac{1}{t_1} + ... + \frac{1}{t_n})$. Those links were introduced by Montesinos in \cite{montesinos1973familia}. 

Let $t = \frac{\alpha}{\beta} $  be a rational number with $\beta > 0$. The \textit{floor} of $t$ is $ \lfloor t \rfloor = max \left\lbrace x \in \Z / x \leq t  \right\rbrace ,$ and the \textit{fractional part} of $t$ is $ \left\lbrace t \right\rbrace = t - \lfloor t \rfloor < 1.$ For $t \neq 1$, define $\hat{t} = \frac{1}{ \left\lbrace \frac{1}{t} \right\rbrace } > 1.$ We also put 
$\left( \frac{\alpha}{\beta} \right)^f = \begin{cases}
\frac{\alpha}{\beta - \alpha} &\text{ if } \frac{\alpha}{\beta} > 0 \\
\frac{\alpha}{\beta + \alpha} &\text{ if } \frac{\alpha}{\beta} < 0
\end{cases}$

Let $L$ be the Montesinos link $M(e;t_1,...,t_n)$. We define $ \varepsilon(L) = e + \displaystyle \sum_{i=1}^n \lfloor \frac{1}{t_i} \rfloor$. The link $L$ is isotopic to $M( \varepsilon(L) ; \hat{t_1}, ... , \hat{t_n})$ (Proposition 3.2 , \citep{champanerkar2015note}). The link $M( \varepsilon(L) ; \hat{t_1}, ... , \hat{t_n})$ is called the \textit{reduced form} of the Montesinos link $L = M(e;t_1,...,t_n)$. 

A complete classification of quasi-alternating Montesinos links is given in \citep{issa2017classification}.

\begin{rmq}
Any Montesinos link is either alternating or almost alternating \cite{abe2010dealternating}.
\end{rmq}

\begin{theo}[Theorem 1, \cite{issa2017classification}]
Let $L=M(e;t_1,...,t_n)$ be a Montesinos link and $M( \varepsilon(L) ; \hat{t_1},...,\hat{t_n})$ its reduced form. Then $L$ is quasi-alternating if and only if 
\begin{enumerate}
\item $\varepsilon(L) > -1$, or
\item $\varepsilon(L) = -1$ and $ \left| \hat{t_i}^f \right| > \hat{t_j}$ for some $i \neq j$, or
\item $\varepsilon(L) < 1-n$, or
\item $\varepsilon(L) = 1-n$ and $ \left| \hat{t_i}^f \right| < \hat{t_j}$ for some $i \neq j$.
\end{enumerate}
\label{th4.1}
\end{theo}
By Theorem 10 in \citep{lickorish1988some}, the only non-alternating and quasi-alternating Montesinos links are those with a reduced form $M( \varepsilon ; \hat{t_1},...,\hat{t_n})$ where $\varepsilon = -1$ and $ \left| \hat{t_i}^f \right| > \hat{t_j}$ for some $i \neq j$, and their reflections. We will show in the following theorem that the latter are almost alternating links which can be constructed iteratively by using the technique developed in Theorem \ref{th3.2}.

\begin{theo}
  Each non-alternating quasi-alternating Montesinos link can be obtained from some Montesinos link $M(-1,\frac{\lambda}{\gamma},\frac{\eta}{\delta})$ by a finite sequence of dealternator extensions, isotopies, rational extensions and likely flype moves.
\end{theo}

\begin{proof}
Let $L$ be a non-alternating quasi-alternating Montesions link. Let $M(\varepsilon ; \frac{\alpha_1}{\beta_1},...,\frac{\alpha_n}{\beta_n})$ be its reduced form. Since $L$ is non-alternating, then $n \geq 3$ as proved in Proposition 3.1 in \cite{champanerkar2015note}. Furthermore, as the link $L$ is quasi-alternating, then  by Theorem \ref{th4.1}  $\varepsilon=-1$ and there exist $i \in [\![1,n]\!]$ and $j \in [\![1,n]\!] \setminus \lbrace i \rbrace$ such that $ \displaystyle \left|  \left( \frac{\alpha_i}{\beta_i} \right)^f \right| > \frac{\alpha_j}{\beta_j}$. Note that we may have $\varepsilon=1-n$. But the latter case corresponds to the mirror image of the former one and then we can restrict ourselves to $\varepsilon=-1$. For convenience, let us denote $\frac{\alpha_i}{\beta_i} = \frac{\lambda}{\gamma}$ and $\frac{\alpha_j}{\beta_j}=\frac{\eta}{\delta}$ and we suppose without loss of generality that $i < j$.\\
Let $D_1$ be the almost alternating link diagram $n \left(-1 + \frac{\gamma}{\lambda} + \frac{\delta}{\eta} \right)$. Let $c_1$ be its dealternator. Note that the link $\mathcal{L}(D_1)$ is the alternating Montesinos link $M \left( -1, \frac{\lambda}{\gamma},\frac{\eta}{\delta} \right)$. By Proposition 4.1 in \cite{champanerkar2015note}, we have $$\det(D_1) = \displaystyle \left| \lambda \eta \left( -1 + \frac{\gamma}{\lambda} + \frac{\delta}{\eta} \right) \right|.$$
Since $ \left|\left(\frac{\lambda}{\gamma}\right)^f \right| > \frac{\eta}{\delta}$, which is equivalent to $\frac{\gamma}{\lambda} + \frac{\delta}{\eta} > 1$, then $\det(D_1) \neq 0$. This shows that $\mathcal{L}(D_1)$ is non-split, hence it is quasi-alternating.\\

Now, starting from $D_1$, we will build the link $L$ by using a finite sequence of operations which preserve the property of being quasi-alternating. Let $k \in [\![1,n]\!] \setminus \lbrace i,j \rbrace$ be a fixed integer.\\
\noindent\underline{Step 1}: If $i>1$ and $k< i$, this step will be skipped. If not, according to whether $i < k < j$ or $j<k$, we use flype moves to slide the dealternator $c_1$ in order to put it between the tangles $\frac{\gamma}{\lambda}$ and $\frac{\delta}{\eta}$ or to the right of the tangle $\frac{\delta}{\eta}$. So we get a new diagram $D_2$ which is equivalent to $D_1$ and is almost alternating. Let $c_2$ denote its dealternator.\\
\noindent\underline{Step 2}: It is easy to see that $\mathcal{L}\left(\left(D_2\right)^{c_2}_{0}\right) = M \left(0;\frac{\lambda}{\gamma},\frac{\eta}{\delta} \right)$ and $\mathcal{L}\left(\left(D_2\right)^{c_2}_{\infty} \right) = (\frac{\gamma}{\lambda}) \# (\frac{\delta}{\eta})$. On the other hand, we have $\det\left(M \left(0;\frac{\lambda}{\gamma},\frac{\eta}{\delta}\right)\right)= \lambda \eta \left(\frac{\gamma}{\lambda} + \frac{\delta}{\eta}\right)$ and $\det\left(d\left(\frac{\gamma}{\lambda}\right) \# d\left( \frac{\delta}{\eta}\right)\right)=\lambda \eta$.  Then $\det\left(\left(D_2\right)^{c_2}_{0}\right) > \det \left(\left(D_2\right)^{c_2}_{\infty}\right)$. Hence by Theorem \ref{th3.2}, the diagram $(D_2)^{c_2\leftarrow-\frac{1}{2}}$ is a quasi-alternating link diagram at each crossing of the extension tangle $-\frac{1}{2}$.\\
\noindent\underline{Step 3}: Now since $-\frac{1}{2}= -1 + \frac{1}{2}$, then we may replace the tangle $-\frac{1}{2}$ in $(D_2)^{c_2\leftarrow-\frac{1}{2}}$ with the tangle $-1 + \frac{1}{2}$ to get an equivalent almost alternating diagram $D_3$ which is either $n\left( -1 + \frac{1}{2} + \frac{\gamma}{\lambda} + \frac{\delta}{\eta} \right)$, or $n\left( \frac{\gamma}{\lambda} + (-1 + \frac{1}{2}) + \frac{\delta}{\eta} \right)$, or $n\left(\frac{\gamma}{\lambda} + \frac{\delta}{\eta} + (-1 + \frac{1}{2}) \right)$ depending on the location of $k$ as mentioned in the first step. One can easily check that $D_3$ is quasi-alternating at each crossing of the tangle $\frac{1}{2}$. Let $c_3$ be the dealternator of $D_3$. By using flype moves if needed, we can assume that $c_3$ is the first rational tangle on the left in the rational parametrization of $D_3$. Note that it is easy to check that $D_3$ is quasi-alternating at each of the crossings $\tau_1$ and $\tau_2$ of the tangle $\frac{1}{2}$.\\
\noindent\underline{Step 4}: Now we are ready to insert the rational tangle $t= \frac{\beta_k}{\alpha_k}$. By Theorem \ref{th1.1} $\displaystyle \left(D_3 \right)^{\tau_1 \leftarrow \frac{t}{1-t}}$ is quasi-alternating at each crossing of the tangle $\frac{t}{1-t}$ and also at $\tau_2$.
Now since $\frac{t}{1-t} * 1 = t$, then we may consider $\tau_2$ as a crossing of the tangle $t$.
Finally we built one of the following links: 
$M \left( -1; \frac{\alpha_k}{\beta_k}, \frac{\lambda}{\gamma}, \frac{\eta}{\delta} \right)$, or $M \left( -1; \frac{\lambda}{\gamma}, \frac{\alpha_k}{\beta_k}, \frac{\eta}{\delta} \right)$, or $M \left( -1; \frac{\eta}{\delta}, \frac{\lambda}{\gamma}, \frac{\alpha_k}{\beta_k} \right)$ depending on the initial position of $t=\frac{\alpha_k}{\beta_k}$ in the rational parametrization of $L$. 
This ends what we call the first loop of the building process. The next loop will start with the ouput of the last one and will allow the insertion of a new tangle $\frac{\alpha_h}{\beta_h}$. The same arguments used in the four last steps work again.\\
Then after $n-2$ loops we will get the wanted rational parametrisation of $L$.

\begin{figure}[H]
\centering
\includegraphics[width=0.9\linewidth]{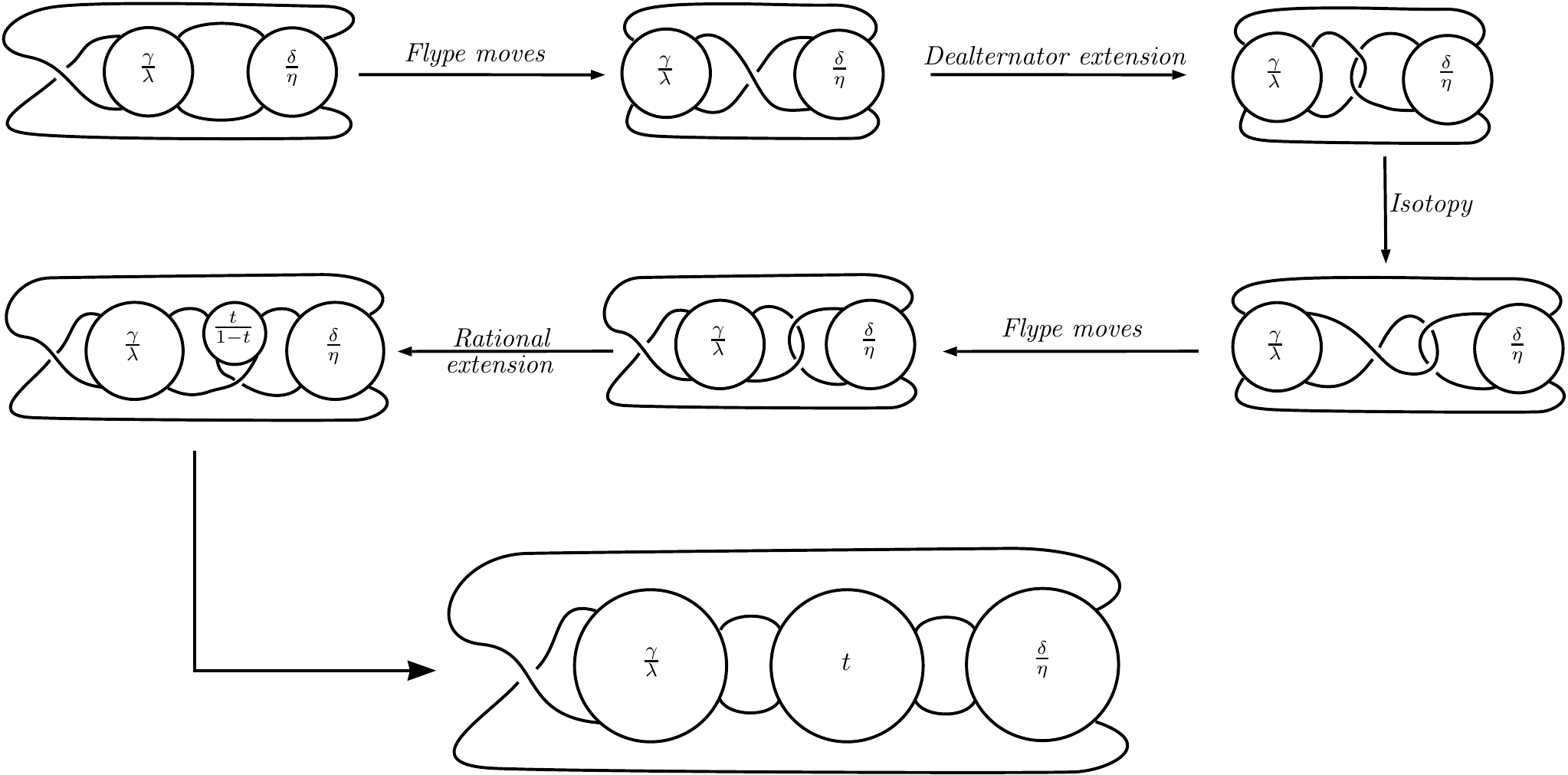}  
\caption{An illustrative diagram of the first iteration where $i<k<j$.}
\label{fig11}
\end{figure}

\end{proof}

\subsection{Generating almost alternating and quasi-alternating links}
Recall the following definitions. A link $L\subset S^3$, other than the unknot, is \textit{prime} if every $2$-sphere in $S^3$ that intersects $L$ transversely at two points bounds on one side of it, a ball that intersects $L$ in precisely one unknotted arc. A diagram $D\subset S^2$, of a link other than the unknot, is a \textit{prime} diagram if any simple closed curve in $S^2$ that meets $D$ transversely at two points bounds, on one side of it, a disk that intersects $D$ in a diagram $U$ of the unknotted ball-arc pair.\\
We will use the Kauffman polynomial $\Lambda_L(a,z)\in\Z[a,a^{-1},z,z^{-1}]$ which is an invariant of regular isotopy for unoriented link diagrams $L$. It satisfies the following relations 

\begin{enumerate}
\item $\Lambda\textsubscript{\includegraphics[scale=0.04]{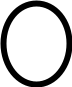}}=1$.
\item $\Lambda\textsubscript{\includegraphics[scale=0.04]{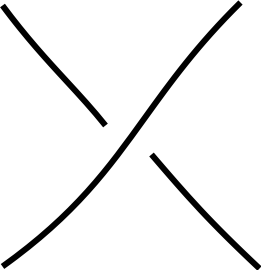}}+\Lambda\textsubscript{\includegraphics[scale=0.04]{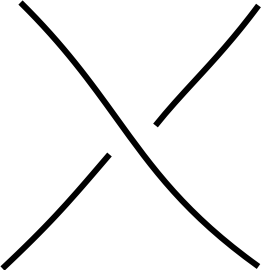}}=z(\Lambda\textsubscript{\includegraphics[scale=0.04]{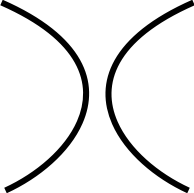}}+\Lambda\textsubscript{\includegraphics[scale=0.04]{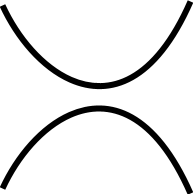}})$.
\item $\Lambda\textsubscript{\includegraphics[scale=0.04]{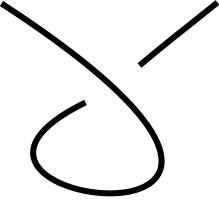}}=a\Lambda\textsubscript{\includegraphics[scale=0.04]{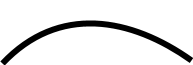}}$.
\item $\Lambda\textsubscript{\includegraphics[scale=0.04]{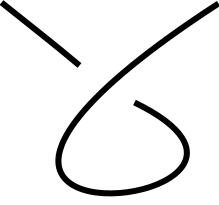}}=a^{-1}\Lambda\textsubscript{\includegraphics[scale=0.04]{ta.png}}$.

\end{enumerate}

If $L$ is a link or a link diagram, then $c(L)$ is the crossing number of $L$.
\begin{prop} Let $T$ be a strongly alternating tangle diagram such that $n(T)$ is prime. If $t$ is a rational tangle diagram such that $0<t<1$ and $\det (N(-t+T))\ne 0$, then the link $N\left( -t + T \right)$ is almost alternating.
\label{prop4.3}
\end{prop}
To prove the proposition we will need the following lemma.
\begin{lem}
Let $T$ be is a strongly alternating tangle diagram such that $n(T)$ is prime. If $t$ is a rational tangle diagram such that $0<t<1$, then
$$\deg_z\Lambda_{n(-t+T)}(a,z) = c\left( n(-t+T) \right) - 2.$$
\label{lem4.4}
\end{lem}

\begin{proof}
We will prove the lemma by induction on the number $k$ of the integer tangles consisting the diagram $t$.\\
\underline{First step}: If $k=1$, since $t<1$ then $t=\frac{1}{p}$ for some integer $p\geq2$. Then $n(-t+T) = n(-\frac{1}{p} + T)$. To prove the result, we will do another induction on $p$: if $p=2$, we apply the skein relation satisfied by the polynomial $L$ at the top crossing of the tangle $-\dfrac{1}{2}$, then we get
$$\Lambda_{n(-\frac{1}{2} + T)} = za\Lambda_{n(T)}+z\Lambda_{n(-1+T)}-\Lambda_{d(T)}.$$
By Theorem 4 in \cite{thistlethwaite1988kauffman}, we have $\deg_z\Lambda_{n(T)} \le c(n(T))-1$. Since $n(T)$ is prime, according to the discussion which follows from Theorem 5 in \cite{thistlethwaite1988kauffman}, we have
$$\deg_z\Lambda_{n(T)} = c(n(T))-1.$$
By applying once again Theorem 4 cited above, we have also $$\deg_z\Lambda_{n(-1+T)}\leq c(n(-1+T))-3=c(n(T))-2\text{, and }\deg_z\Lambda_{d(T)} \leq c(d(T))-1.$$
Furthermore we know that $c(n(T)) = c(d(T))$. Finally  we get that 
\begin{align*}
\deg_z\Lambda_{n(-\frac{1}{2} + T)} &= \deg_z \left(za\Lambda_{n(T)}\right) \\
&= 1 + c(n(T))-1 = c(n(-\frac{1}{2} + T))-2.
\end{align*}
Now if the result holds up to an integer $p$, $p\ge 2$, by applying a new time the skein relation at the top crossing of the integer tangle $-\dfrac{1}{p+1}$ we obtain that $\deg_z\Lambda_{n(-\frac{1}{p+1} + T)}=c(n(-\frac{1}{p+1} + T))-2$.\\
\underline{Second step}: Let $k\geq 1$ be a fixed integer and suppose that the lemma holds up to $k$. Let $t$ be a rational tangle diagram consisting of $k+1$ integer tangles. We can write $t=\left[0,a_1,...,a_k,p\right]$. Note that if $p=1$, by adding $p$ with the integer tangle $a_k$ the tangle $t$ becomes a rational tangle consisting of $k$ integer tangles for which the result holds by the induction hypothesis. So we can suppose that $p\geq2$. Without loss of generality, we may assume that $k$ is even, in which case the integer tangle $p$ will be vertical. We do another induction on $p$. If $p=2$, then $n(-t+T)=n(-\left[0,a_1,...,a_k,2\right]+T)$. By applying a new time the skein relation at the top crossing of the integer tangle $2$ we get

$$
\Lambda_{n(-\left[0,a_1,...,a_k,2\right]+T)}=z\Lambda_{n(-\left[0,a_1,...,a_k+1\right]+T)}+za\Lambda_{n(-\left[0,a_1,...,a_k\right]+T)}
-a^{a_k}\Lambda_{n(-\left[0,a_1,...,a_{k-1}\right]+T)}.
$$

By the induction hypothesis, we get 
\begin{align*}
\deg_z\Lambda_{n(-\left[0,a_1,...,a_k,2\right]+T)}&=\deg_z\left(z\Lambda_{n(-\left[0,a_1,...,a_k+1\right]+T)}\right)\\
&=1+c(n(-\left[0,a_1,...,a_k+1\right]+T))-2\\
&=c(n(-\left[0,a_1,...,a_k,2\right]+T)) - 2.
\end{align*}
Now if we assume that the result holds for any integer up to $p$, $p\ge 2$, it easy to proof the result for $p+1$ by writing the skein formula and by using the induction hypothesis, and the proof is done.
\end{proof}
\begin{proof}[Proof of Proposition \ref{prop4.3}]
  We prove that $N(-t + T)$ is not alternating and we provide an almost alternating link diagram of it.\\
  Suppose that the link $N(-t + T)$ is alternating. Since $\det(-t+T)\ne 0$, then $N(-t+T)$ has a connected, reduced and alternating diagram $D$. By using Theorem 1 in \cite{thistlethwaite1987spanning}, we have $Br\langle D\rangle = 4c(D)$ where $Br\langle D\rangle$ is the breadth of $\langle D\rangle$, i.e. the difference between the maximal degree and the minimal degree of the indeterminate that occur in the Kauffman bracket polynomial $\langle D\rangle$. On the other hand, by using the tangle isotopy $-t = -1 + (\frac{1-t}{t} * 1)$, we get that $D$ is equivalent to the dealternator reduced and dealternator connected diagram $n(-1 + (\frac{1-t}{t} * 1) + T)$ as shown in Fig. \ref{fig16}.
  \begin{figure}[H]
\centering
\includegraphics[width=0.5\linewidth]{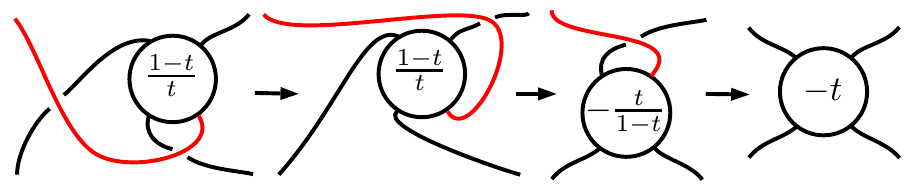}  
\caption{A equivalence of $-1 + (\frac{1-t}{t} * 1)$ and $-t$.}
\label{fig16}
\end{figure}
  \begin{figure}[H]
\centering
\includegraphics[width=0.6\linewidth]{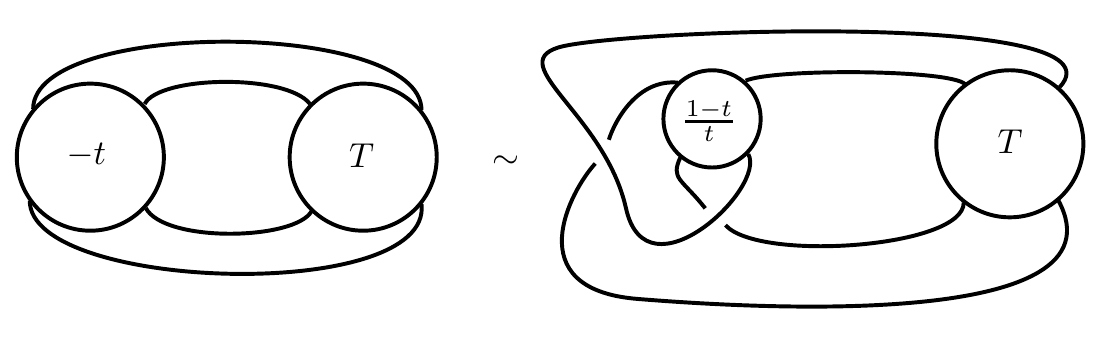} 
\caption{The link diagram $n(-t+T)$ (on the left) and the equivalent almost-alternating, dealternator reduced, and dealternator connected diagram $n(-1 + (\frac{1-t}{t} * 1) + T)$ (on the right).} 
\label{fig17}
\end{figure}
By Theorem 4.4 in \cite{adams1992almost}, we have $$Br\langle n(-1+(\frac{1-t}{t}*1)+T)\rangle \leq 4\left(c(n(-1+(\frac{1-t}{t}*1)+T))-3\right).$$ 
We can easily see from Fig. \ref{fig16} that $c( n(-1 + (\frac{1-t}{t} * 1) + T)) = 1 + c( n(-t + T))$. Hence, $c(D) \leq c\left( n(-t + T) \right) -2$. By Theorem 4 in \cite{thistlethwaite1988kauffman}, we have $\deg_z\Lambda_{D} \leq c(D) - 1$. Hence, by Lemma \ref{lem4.4} we have $c(n(-t + T))-2 \leq c(D) - 1$. This is absurd since $c(D) \leq c( n(-t + T)) -2$.
\end{proof}

\begin{theo}
Let $T$ be a strongly alternating tangle diagram, $k$ be an integer, $k\geq 1$, and $\{t_1,\dots ,t_k\}$ be a family of rational tangle diagrams such that $0<t_i<1$, for each $i$. If the link $N(-1 + T)$ is quasi-alternating, $n(T)$ is prime, and $\det\left( n(T) \right) > \det\left( d(T) \right)$, then the link $N(-1 + t_1 + ... + t_k + T)$ is almost alternating and quasi-alternating.
\label{th4.5}
\end{theo}

Denote by $T_k = \frac{1}{2} + ... + \frac{1}{2} + T$ the sum of $k$ copies of the rational tangle $\frac{1}{2}$ and the tangle $T$.

\begin{rmq}
The following results easily follow from the assumptions assumed in the statement of the theorem above. For each $k \geq 1$, the tangle diagram $T_k$ is strongly alternating and locally unknotted. Since $d(T_k)$ is composite then by using Lemma \ref{lem2.2}, the link diagram $n(T_k)$ is prime. Furthermore by simple calculations, one can prove the following relations
$$\det(d(T_{k}))=2^k\det(d(T))\,\,\text{and}\ \ \det(n(T_{k}))=2^k\det(n(T))+k2^{k-1}\det(d(T)).$$
Hence $\det (n(T_{k}))>\det (d(T_{k}))$.
\end{rmq}

\begin{lem}
  Let $T$ be a strongly alternating tangle diagram such that $n(T)$ is prime, the link $N(-1+T)$ is quasi-alternating and $\det (n(T))>\det (d(T))$. Then for each $k$, the link $N(-1+T_k)$ is almost alternating and quasi-alternating. 
\end{lem}
\begin{proof} We will do an induction on $k$.\\
  If $k=1$: we must show the result for the link $N(-1+T_1)$. Note that the rational tangles $-\frac{1}{2}$ and $-1+\frac{1}{2}$ are the same then the links $N(-\frac{1}{2}+T)$ and $N(-1+(\frac{1}{2}+T))=N(-1+T_1)$ are equivalent. The tangle diagram $T$ is strongly alternating and $n(T)$ is prime. Since the rational tangle $\frac{1}{2}$ is less than one, then by Proposition \ref{prop4.3} the link  $N(-\frac{1}{2}+T)$ is almost alternating.\\
  On the other hand, $n(-1+T)$ is an almost alternating diagram of a quasi-alternating link. Remembering that $n(T)$ and $d(T)$ are exactly the smoothings of $n(-1+T)$ at $c$, since $\det (n(T))>\det (d(T))$, then by using Theorem \ref{th3.2}, the dealternator extension of the diagram $n(-1+T)$ by $-\frac{1}{2}$ which is $N(-\frac{1}{2}+T)$ is quasi-alternating.\\
Now, assume that the result is true for $k-1$. That is $N(-1+T_{k-1})$ is almost alternating and quasi-alternating. By the remark above, the tangle $T_{k-1}$ satisfies the same conditions as $T$. By applying the same arguments used for the case $k=1$ with $T_k$ instead $T$, we obtain that the link $N(-\frac{1}{2}+T_{k-1})=N(-1+\frac{1}{2}+T_{k-1})=N(-1+T_{k})$ is both quasi-alternating and almost alternating.
\end{proof}

\begin{proof}[Proof of Theorem \ref{th4.5}]
We start by noting that the link diagram $n(-1+T_k)$ is quasi-alternating at each crossing of its rational tangles $\frac{1}{2}$. Indeed, for any integer $j$, $1\le j\le k$, the two smoothings of $n(-1+T_k)$ at the top crossing of the $j$-th tangle $\frac{1}{2}$ are exactly $n(-1+T_{k-1})$ and $n(T_{k-1})$ which are both quasi-alternating. Moreover $\det(n(-1+T_{k-1}))=\det(n(-1+T_k))+\det(n(T_{k-1}))$. By a similar argument one can show that the bottom crossing of the $j$-th tangle $\frac{1}{2}$ is also quasi-alternating.\\
Now, for each $i$, $1\le i\le k$,  we extend the link diagram $n(-1+T_k)$ at the top crossing of the $i$-th tangle $\frac{1}{2}$ by the rational tangle by $\frac{t_i}{1-t_i}$. By using the rational tangle equivalence $\frac{t_i}{1-t_i}*1 = t_i$, we see that those extensions provide the link diagram $n(-1+t_1+...+t_k+T)$ which is then quasi-alternating by Theorem \ref{th1.1}.\\
To see that $n(-1+t_1+...+t_k+T)$ is almost alternating, one can apply Proposition \ref{prop4.3} by considering the rational tangle $(1-t_1)<1$ and the tangle $t_2+...+t_k+T$.
\end{proof}

\begin{rmq} After testing on several examples, we suspect the following result. If $T$ is a strongly alternating tangle diagram such that $n(T)$ is prime, then $$\det \left( n(T) \right) > \det \left( d(T) \right).$$ If this is true, it will allow to reduce the assumptions of the Theorem \ref{th4.5}.
\end{rmq}

\begin{cor} 
Let $T$ be a strongly alternating tangle diagram, let $k$ be an integer, $k\geq 1$, and let $\{t_1,\dots ,t_k\}$ be a family of rational tangle diagrams such that $0<t_i<1$, for each $i$. If the link $N(-1 + T)$ is quasi-alternating, $d(T)$ is prime, and $\det\left( d(T) \right) > \det\left( n(T) \right)$, then the link $N(-1 + t_1 + ... + t_k + \frac{1}{T}_{cc})$ is almost alternating and quasi-alternating.
\end{cor}

\begin{proof}
The main observation is that the link diagram $n(-1 + \frac{1}{T}_{cc})$ is equivalent to $n(-1+T)$ up to mirror image as one can see in Fig. \ref{fig12}. Hence the link $N(-1 + \frac{1}{T}_{cc})$ is quasi-alternating.

\begin{figure}[H]
\centering
\includegraphics[width=0.6\linewidth]{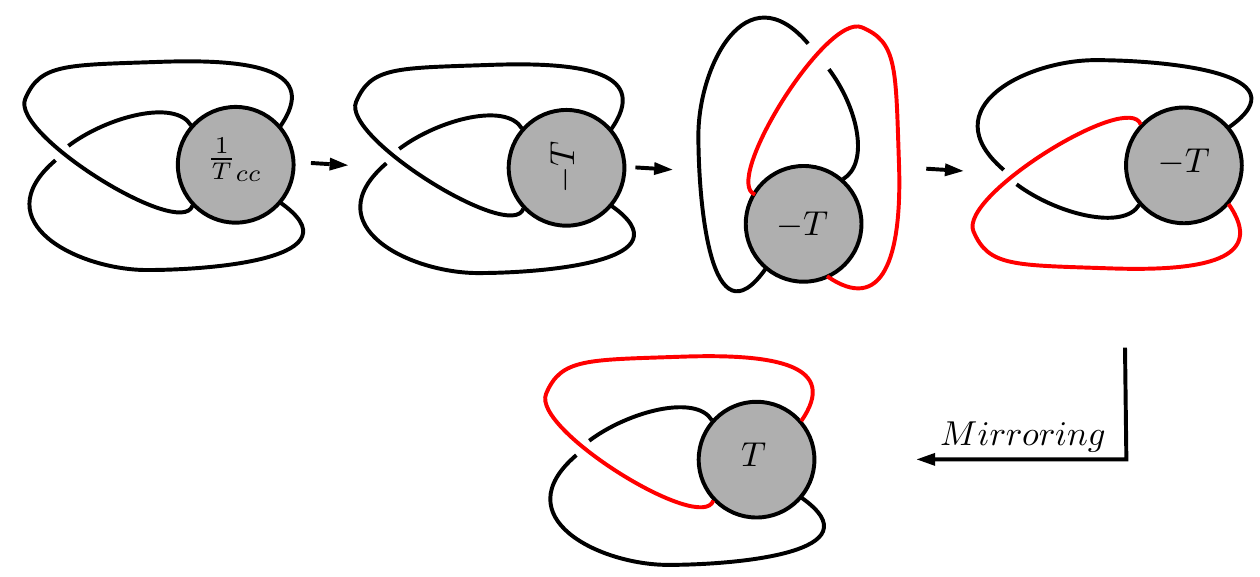} 
\caption{Equivalence between $n(-1 + T)$ and $n(-1 + \frac{1}{T}_{cc})$.} 
\label{fig12}
\end{figure}

On the other hand, $\frac{1}{T}_{cc}$ is strongly alternating and $n(\frac{1}{T}_{cc})=d(-T)=-d(T)$. Then
$$\det(n(\frac{1}{T}_{cc}))=\det(d(T))>\det(n(T))=\det(-(-n(T))=\det(n(-T))=\det(d(\frac{1}{T}_{cc})).$$
Then we apply Theorem \ref{th4.5} to the tangle $\frac{1}{T}_{cc}$ instead of $T$.
\end{proof}

\begin{empl}
We consider the quasi-alternating and almost alternating link diagram on the left in Fig. \ref{exmpl16}. By applying Theorem \ref{th4.5} we get the quasi-alternating and almost alternating link whose diagram is on the right in the same figure.
 
 \begin{figure}[H]
\centering
\includegraphics[width=0.6\linewidth]{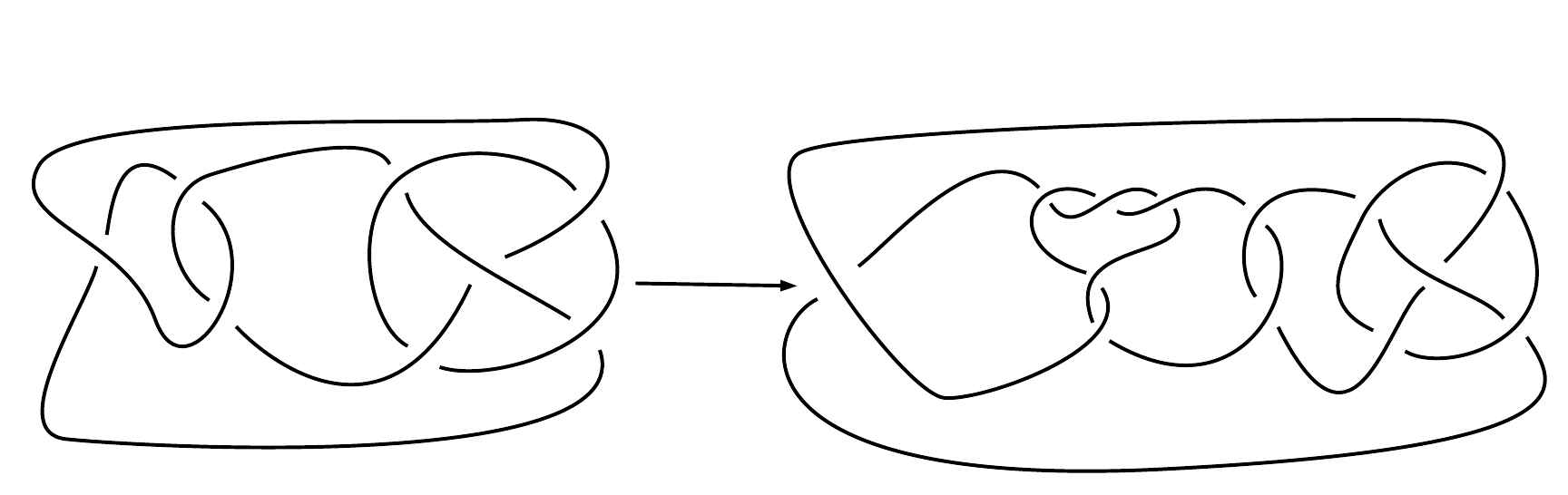} 
\caption{} 
\label{exmpl16}
\end{figure}
\end{empl}

\begin{theo}
Every connected, locally unknotted, reduced alternating tangle diagram generates infinitely many quasi alternating and almost alternating links.
\label{th4.6}
\end{theo}

\begin{proof}
  Let $ t= \left[ 0,a_1,...,a_k\right]$ be a rational tangle such that $k \geq 2$ is an even integer and $T$ be a connected, locally unknotted, reduced alternating tangle diagram. We denote by $T \star t$ the tangle diagram depicted in Fig. \ref{fig21bis}.

\begin{figure}[!htbp]
\centering
\includegraphics[width=0.3\linewidth]{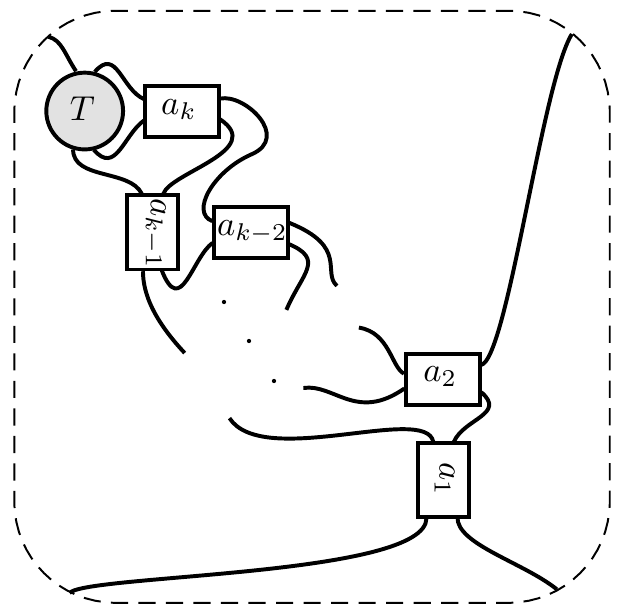}  
\caption{The tangle diagram $T \star t$.}
\label{fig21bis}
\end{figure}

We note that the link diagram $n(-t+T \star t)$ equivalent to $n(T)$ or $n(T_h)$ according to the parity of $a_k$ (the links $N(T)$ and $N(T_h)$ are the same). Fig. \ref{fig13} exhibits that equivalence for a particular rational tangle $\left[ 0,n,2m \right]$. The link $N(-t+T \star t)$ is alternating. It is non-split by connectedness of $T$. By using the rational tangle equivalence $-t= -1+(\frac{1-t}{t}*1)$, the link diagram $n(-t+T \star t)$ is equivalent to the almost alternating, dealternator connected, and dealternator reduced link diagram $n\left( -1+(\frac{1-t}{t}*1) + (T \star t) \right)$. Denote by $S$ the tangle diagram $(\frac{1-t}{t}*1) + (T \star t)$. It is clear that $S$ is strongly alternating and Lemma \ref{lem2.2} provides that $n(S)$ is prime. Furthermore, simple calculations show that $\det(n(S)) > \det(d(S))$. We can then apply Theorem \ref{th4.5} to generate infinitely many quasi-alternating and almost alternating links starting with the link diagram $n(-1+S)$, which represents the non-split alternating link $N(T)$.
\end{proof}

\begin{figure}[!htbp]
\centering
\includegraphics[width=0.7\linewidth]{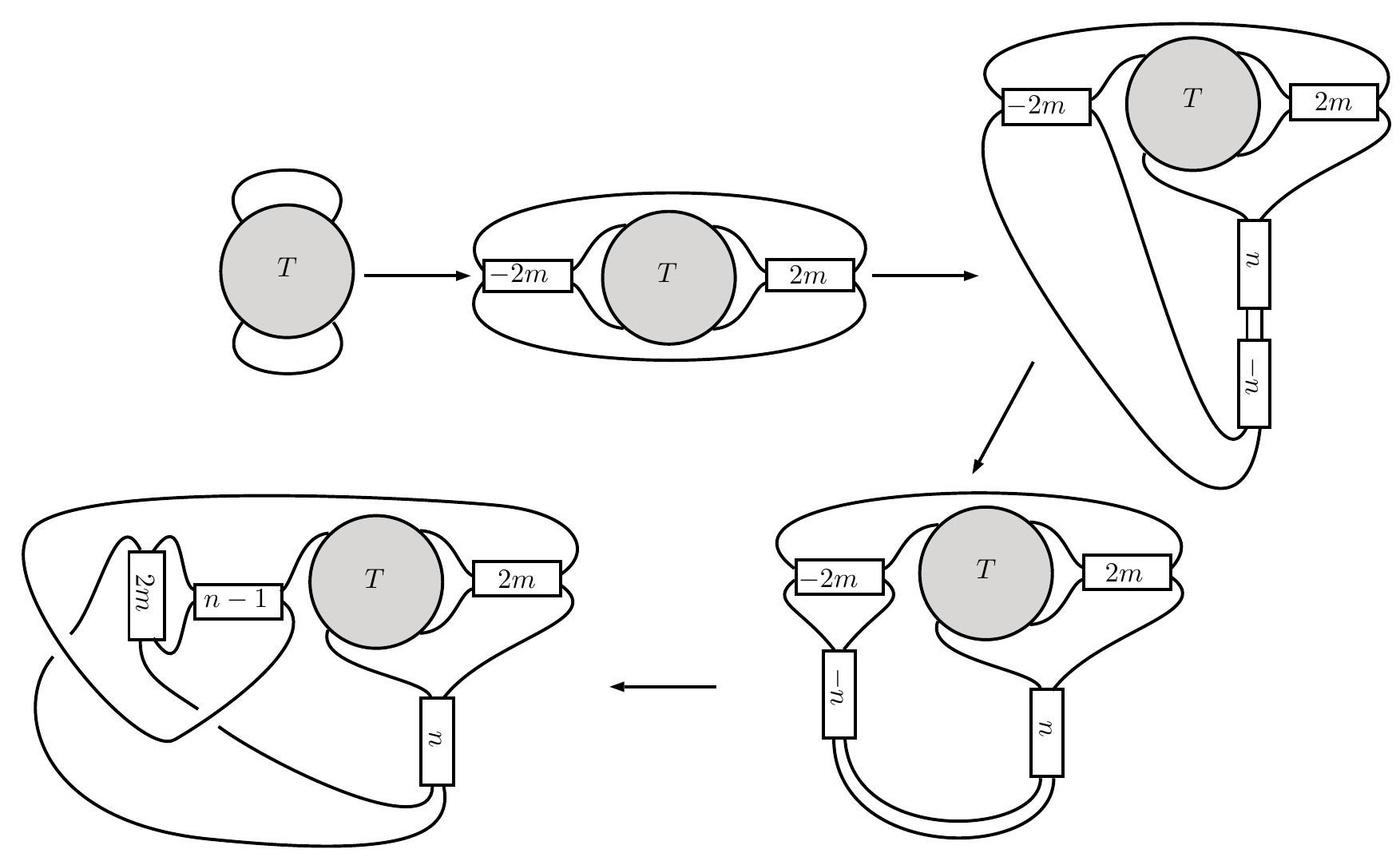}  
\caption{The link diagram $n(T)$ is equivalent to $n(-\left[ 0,n,2m \right] +T\star[0,n,2m])$. This is equivalent to the almost alternating link diagram $ n(-1 + \left[ 0,1,n-1,2m \right] + T \star \left[ 0,n,2m \right]) $.}
\label{fig13}
\end{figure}

\section{Crossing numbers and determinants of the generated links}
Qazaqzeh et al. showed that the cossing number of any alternating (non-split) link is less than its determinant (see Proposition 2.2 in \cite{qazaqzeh2013remark}). Then following many verifications for some known families of quasi-alternating links, they stated Conjecture \ref{conj1}. Our aim in this section is to check that the conjecture is satisfied by all links provided by dealternator extensions.

\begin{prop}
Let $D$ be an almost alternating link diagram with dealternator $c$ such that $\mathcal{L}(D)$ is quasi-alternating and $c(D) \leq \det(D)$. Let $D^{c \leftarrow \omega}$, where $\omega$ is equal to either $-\frac{1}{2} \text{ or } -2$, be the quasi-alternating dealternator extension of $D$ obtained by Theorem \ref{th3.2}. Then every rational extension $\mathcal{D}$ of $D^{c \leftarrow\omega}$ at $c$ satisfies $c(\mathcal{D}) \leq \det(\mathcal{D})$.
\label{prop5.1}
\end{prop}

\begin{proof}
  We first consider the case where the quasi-alternating dealternator extension given by Theorem \ref{th3.2} is $D^{c \leftarrow-\frac{1}{2}}$. Note that this occurs if and only if $ \det\left( D^{c}_{0} \right) > \det\left( D^{c}_{\infty} \right)$.\\
We have $c\left( D^{c \leftarrow-\frac{1}{2}} \right) = c(D) +1 \leq \det(D)+1$. Now since $D^{c \leftarrow-\frac{1}{2}}$ is by assumption quasi-alternating at $c$, then 
$$\det\left( D^{c \leftarrow-\frac{1}{2}} \right) =\det((D^{c \leftarrow-\frac{1}{2}})^c_\infty)+\det((D^{c \leftarrow-\frac{1}{2}})^c_0)= \det\left( D\right) + \det\left( D^{c}_{0} \right).$$
Since $\det\left( D^{c}_{0} \right) > \det\left( D^{c}_{\infty} \right) \geq 0$, then 

$$c\left( D^{c \leftarrow-\frac{1}{2}} \right)  = c(D) +1 \leq \det(D)+1 \leq \det\left( D^{c \leftarrow-\frac{1}{2}} \right)$$

Now if $\mathcal{D}$ is a rational extension of $D^{c \leftarrow-\frac{1}{2}}$ at $c$, which is a quasi-alternating crossing in $D^{c \leftarrow-\frac{1}{2}}$, then by Theorem 2.3 in \citep{qazaqzeh2013remark} we get that $c(\mathcal{D}) \leq \det(\mathcal{D})$.
A similar argument gives the result in the case where the quasi-alternating dealternator extension given by Theorem \ref{th3.2} is $D^{c \leftarrow-2}$. 
\end{proof} 

\begin{prop}
Let $D$ be an almost alternating link diagram, dealternator reduced and dealternator connected, with dealternator $c$ and such that $\mathcal{L}(D)$ is quasi-alternating. Let $D^{c \leftarrow\omega}$, where $\omega$ is equal either to $-\frac{1}{2} \text{ or } -2$, be the quasi-alternating dealternator extension of $D$ obtained by Theorem \ref{th3.2}. Then every rational extension $\mathcal{D}$ of $D^{c \leftarrow\omega}$ at $c$ satisfies $c(\mathcal{D}) \leq \det(\mathcal{D})$.
\label{prop5.2}
\end{prop}

\begin{proof}
We first consider the case where the quasi-alternating dealternator extension given by Theorem \ref{th3.2} is $D^{c \leftarrow-\frac{1}{2}}$. Note that this case occurs when $\det \left( D^{c}_{0} \right) > \det \left( D^{c}_{\infty} \right)$.
Suppose that $c \left( D^{c \leftarrow-\frac{1}{2}} \right) > \det \left( D^{c \leftarrow-\frac{1}{2}} \right)$, then

\begin{align}
c(D)+1 = c\left( D^{c}_{\infty} \right) + 2 > 2 \det \left( D^{c}_{0} \right) - \det \left( D^{c}_{\infty} \right).
\label{cond3}
\end{align}

But since $D^{c}_{\infty}$ is alternating and reduced, then by Proposition 2.2 in \citep{qazaqzeh2013remark} we get that $c\left( D^{c}_{\infty} \right) \leq \det\left( D^{c}_{\infty} \right)$. Hence, (\ref{cond3}) implies that $$\det\left( D^{c}_{\infty} \right) + 2 > 2 \det\left( D^{c}_{0} \right) - \det\left( D^{c}_{\infty} \right).$$
This is equivalent to
$$\det\left( D^{c}_{\infty} \right) \geq \det\left( D^{c}_{0} \right).$$
We get a contradiction. Then we conclude that $c \left( D^{c \leftarrow-\frac{1}{2}} \right) \leq \det \left( D^{c \leftarrow-\frac{1}{2}} \right)$.
Now if $\mathcal{D}$ is a rational extension of $D^{c \leftarrow-\frac{1}{2}}$ at $c$, which is a quasi-alternating crossing in $D^{c \leftarrow-\frac{1}{2}}$, then by Theorem 2.3 in \citep{qazaqzeh2013remark} we get that $c(\mathcal{D}) \leq \det(\mathcal{D})$.\\
 We prove in a similar way the result when the quasi-alternating dealternator extension given by Theorem \ref{th3.2} is $D^{c \leftarrow-2}$.
\end{proof}

\begin{cor}
Every quasi-alternating and almost alternating link provided by Theorem \ref{th4.5} or Theorem \ref{th4.6} satisfies Conjecture \ref{conj1}.

\label{mycor4}
\end{cor}

\begin{proof}
The links provided by Theorem \ref{th4.5} and Theorem \ref{th4.6} are obtained by dealternator extensions of dealternator reduced and dealternator connected diagrams which represent quasi-alternating links. Then, they satisfy Conjecture \ref{conj1} by Proposition \ref{prop5.2}.
\end{proof}

\begin{prop}
Let $T$ be a strongly alternating tangle diagram such that $n(T)$ is prime. Let $t$ be a rational tangle, $0 < t < 1$. If the link $N(-t+T)$ is quasi-alternating, then $c(N(-t+T)) \leq \det(N(-t+T))$.
\label{prop5.3}
\end{prop}
\begin{proof}
By Lemma \ref{lem4.4}, we have $\deg_z\Lambda_{n(-t+T)}=c(n(-t+T))-2$. By using Theorem 1.2. in \cite{teragaito2015quasi}, we have $\deg_z\Lambda_{n(-t+T)}=c(n(-t+T))-2 \leq \det(N(-t+T)) - 2$. Then the result follows easily.
\end{proof}
\begin{cor}
Every quasi-alternating Montesinos link $L$ satisfies Conjecture \ref{conj1}.
\label{mycor3}
\end{cor}

\begin{proof}
The result for quasi-alternating Montesinos links which are also alternating is provided by Proposition 2.2 in \citep{qazaqzeh2013remark}.

Let $L$ be a non-alternating, quasi-alternating Montesions link. Then by Theorem \ref{th4.1} there exist an integer $n \geq 3$ and an ordered set of rationals $\lbrace \frac{\alpha_k}{\beta_k} \rbrace_{1 \leq k \leq n}$ all greater than one where $L$ is isotopic, up to mirror image, to the link $M \left( -1;\frac{\alpha_1}{\beta_1}, ... , \frac{\alpha_n}{\beta_n} \right)$ which is the same as $N\left(-1 + \frac{\beta_1}{\alpha_1}, ... , \frac{\beta_n}{\alpha_n} \right)$. Put $t=\frac{\alpha_1 -\beta_1}{\alpha_1}$ and $T=\frac{\beta_2}{\alpha_2}+...+\frac{\beta_n}{\alpha_n}$. The tangle $t$ is a rational tangle and we have $-t=-1+\frac{\beta_1}{\alpha_1}$. On the other hand, the tangle diagram $T$ is strongly alternating and $n(T)$ is prime by Lemma \ref{lem2.2}. Now since the link $L$ is equivalent to $N(-t+T)$, the result follows by Proposition \ref{prop5.3}.
\end{proof}

\section{The converse of Theorem \ref{th1.1} is false}
Let $D$ be a link diagram in the plane. Suppose that there exists a disk in the plane meeting $D$ transversely four times and enclosing a tangle diagram $T$. We say that $T$ is \textit{embedded} in $D$.

Take an almost alternating diagram $D$ such that $\mathcal{L}(D)$ is quasi-alternating and denote by $c$ the dealternator of $D$. Without loss of generality, let us assume that $\det (D^{c}_{0}) > \det (D^{c}_{\infty})$. By Theorem \ref{th3.2} the diagram $D^{c\leftarrow \frac{1}{-2}}$ is quasi-alternating at each crossing of the tangle $\frac{1}{-2}$. When replacing this tangle by one of its crossings, we get back $D$.  This diagram will not be quasi-alternating at $c$ as one can see by using Corollary \ref{cor2}. This observation supports the following question asked by Chbili and Qazaqzeh in \cite{qazaqzeh2015characterization}.

\begin{qst}[Question 1, \citep{qazaqzeh2015characterization}]
Let $D$ be a quasi-alternating diagram at a crossing $c$ that is a part of a rational tangle diagram $t$ embedded in $D$. Let $D'$ be the link diagram obtained by replacing the projection disk of $t$ by the single crossing $c$. Is the link $\mathcal{L}(D')$ quasi-alternating ?
\label{qst1}
\end{qst}
In fact Question \ref{qst1} asks if the converse of Theorem \ref{th1.1} is true. In order to give an answer, we exhibit some almost alternating diagrams representing non-quasi-alternating links whose dealternator extensions yield quasi-alternating diagrams. 

\begin{lem}
Let $L=N(T+S)$ be a semi-alternating link. Then $L$ admits a dealternator reduced and dealternator connected almost-alternating diagram $D$. Furthermore, if we denote by $c$ the dealternator of $D$, then the link $\mathcal{L}\left(D^{c\leftarrow \frac{1}{-2}} \right)$ is a non-split alternating link.
\label{lem6.1}
\end{lem}
\begin{proof}
Assume without loss of generality that $T$ is of type 1 and $S$ is of type 2. Let $S'$ denote the type 1 tangle diagram $(-\frac{1}{S}_{cc})_h$ and $D$ denote the link diagram $n(-1+(S'+1)*(T+1))$. We have $\mathcal{L}(D) = N(T+S)$ as shown in Fig. \ref{fig14}. Furthermore, the diagram $D$ is almost-alternating, dealternator reduced, and dealternator connected. 

\begin{figure}[!htbp]
\centering
\includegraphics[width=0.4\linewidth]{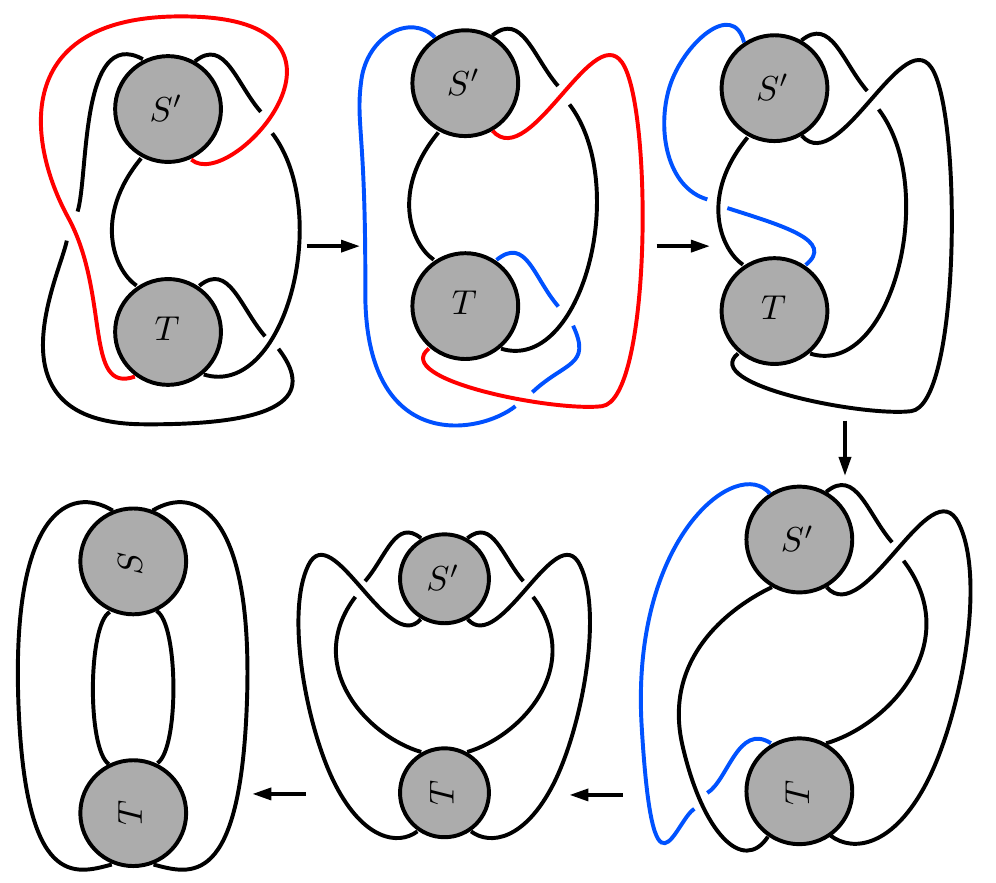} 
\caption{The link diagram $n(-1 +(S'+1)*(T+1))$ is equivalent to the semi alternating diagram $n(T+S)$.} 
\label{fig14}
\end{figure}

Denote by $c$ the dealternator of $D$. Fig. \ref{fig15} shows that $D^{c\leftarrow \frac{1}{-2}}$ is equivalent to a connected alternating diagram. Hence $D^{c\leftarrow \frac{1}{-2}}$ is a non-split alternating link. 

\begin{figure}[!htbp]
\centering
\includegraphics[width=0.6\linewidth]{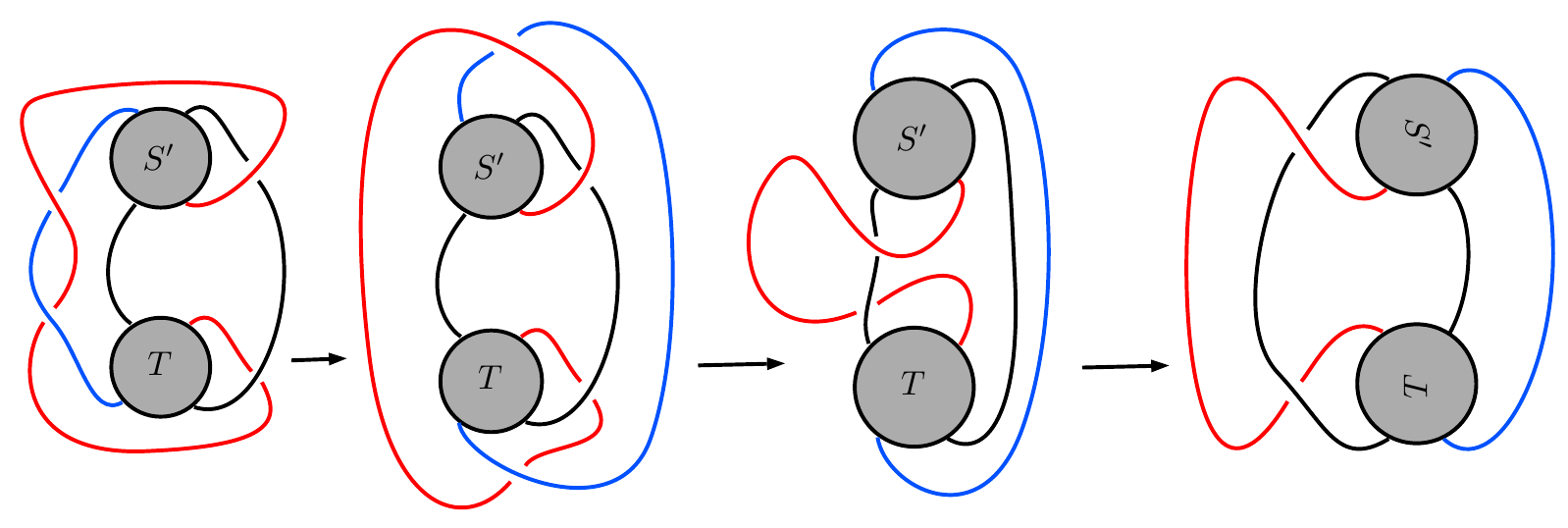} 
\caption{The link diagram $D^{c\leftarrow \frac{1}{-2}}$ is equivalent to an alternating diagram.} 
\label{fig15}
\end{figure}
\end{proof}

\begin{prop}
There exists a link diagram $D$, quasi-alternating at some crossing $c$ that belongs to an embedded rational tangle $t$ such that, the replacement of $t$ with the single crossing $c$ yields a non-quasi-alternating link.
\label{prop6.2}
\end{prop}

\begin{proof}
Let  $T$ and $S$ be strongly alternating tangle diagrams of type 1 and type 2 respectively and put $S'=(-\frac{1}{S}_{cc})_h$. Take $D$ to be the link diagram $n(-\frac{1}{3}+(S'+1)*(T+1))$ and denote by $c$ the top crossing of the vertical tangle diagram $-\frac{1}{3}$. Denote by $d$ the almost alternating diagram $n(-1+(S'+1)*(T+1))$ and let $c'$ denote its dealternator. Clearly $D$ is exactly the diagram $d^{c'\leftarrow\frac{1}{-3}}$. Lemma \ref{lem6.1} shows that the link $\mathcal{L}(d^{c'\leftarrow\frac{-1}{2}}) = N(\frac{-1}{2} + (S'+1)*(T+1))$ is alternating non-split and then it is quasi-alternating. By Corollary \ref{cor5} $D$ is quasi-alternating at the crossing $c$.

If we replace in $D$ the vertical tangle diagram $-\frac{1}{3}$ with the single crossing $c$ we obtain the diagram $d$. Lemma \ref{lem6.1} shows that the link $\mathcal{L}(d)$ is semi-alternating, hence non-quasi-alternating.
\end{proof}

Proposition \ref{prop6.2} answers negatively Question \ref{qst1} by exhibiting an infinite family of counterexamples.

\begin{rmq}
Note that Proposition \ref{prop6.2} shows also that the converse of Theorem 3.3 in \cite{qazaqzeh2015characterization} is false.
\end{rmq}

To complete this work, we ask some questions.

  Let $t$ be a rational tangle diagram, $0<t<1$, and $T$ be a strongly alternating tangle diagram. The Proposition \ref{prop3.3}, provides a necessary (but not sufficient) condition for the link $L=N(-t+T)$ to be quasi-alternating, namely $\frac{\det \left( n(T) \right)}{\det \left( d(T) \right)} > t$. By analogy with the characterization of quasi-alternating Montesinos links in Theorem \ref{th4.1}, we can expect a characterization of quasi-alternateness of links $N(-t+T)$ by some algebraic relation between $t$ and $\frac{\det \left( n(T) \right)}{\det \left( d(T) \right)}$. Hence the following question

\begin{qst}
Do the rational numbers $t$ and $\dfrac{\det \left( n(T) \right)}{\det \left( d(T) \right)}$ suffice to characterize the quasi-alternateness of the link $N(-t+T)$?
\end{qst}

On the other hand, Chbili and Qazaqzeh recently stated a conjecture about quasi-alternating links depending on the coefficients of their Jones polynomials. Write the Jones polynomial of a link $L$ as $V_L(t) = \displaystyle \sum_{i=0}^{m} a_it^i$, where $m \geq 0$, $a_0 \neq 0$, and $a_m \neq 0$. The conjecture was formulated as follows.
\begin{conj}[Conjecture  2.3, \cite{chbili2019jones}]
If $L$ is a prime quasi-alternating link, other than $(2, n)$-torus link, then the coefficients of the Jones polynomial of $L$ satisfy $a_ia_{i+1} < 0$ for all $0 \leq i \leq m-1$.
\label{conj3}
\end{conj}
\begin{qst}
Do the quasi-alternating links arising as dealternator extensions satisfy Conjecture \ref{conj3}?
\end{qst}
\noindent\textbf{Acknowledgements}:
We thank the Referee for his/her comments which allowed to improve the paper.
\nocite{*}
\bibliographystyle{plain}
\bibliography{biblio.bib}
\end{document}